\documentclass[10pt,reqno]{amsart}
\usepackage{amsmath,amssymb,latexsym,esint,cite,mathrsfs,slashed, dsfont}
\usepackage{verbatim,cite,relsize,color,soul}
\usepackage[latin1]{inputenc}
\usepackage{microtype}
\usepackage{color,enumitem,graphicx}
\usepackage[colorlinks=true,urlcolor=blue, citecolor=red,linkcolor=blue,
linktocpage,pdfpagelabels, bookmarksnumbered,bookmarksopen]{hyperref}
\usepackage[hyperpageref]{backref}
\usepackage[english]{babel}
\usepackage{tikz}
\usepackage[font=small,skip=5pt]{caption}
\usepackage{enumitem}

\usepackage{amsrefs}
\renewcommand{\eprint}[1]{\url{#1}} 

\usepackage{geometry}
\numberwithin{equation}{section}
\newtheorem{theorem}{Theorem}[section]
\newtheorem{lemma}[theorem]{Lemma}

\newtheorem{proposition}[theorem]{Proposition}
\newtheorem{corollary}[theorem]{Corollary}

\newtheoremstyle{remarkstyle}
{}{}{\itshape}{ }{\bfseries}{.}{ }{\thmname{#1}\thmnumber{ #2}\thmnote{ (#3)}}
\theoremstyle{remarkstyle}
\newtheorem{remark}{Remark}[section]
\newtheorem{definition}{Definition}[section]

\newcommand{\Z}{\mathbb Z}

\newcommand{\R}{\mathbb R}
\newcommand{\C}{\mathbb C}
\newcommand{\Sb}{\mathbb S}
\newcommand{\Kc}{\mathcal K}
\newcommand{\Hc}{\mathcal H}
\newcommand{\Mca}{\mathcal M}

\newcommand{\vareps}{\varepsilon}
\newcommand{\nnabla}{\slashed{\nabla}}

\DeclareMathOperator*{\opt}{opt}
\DeclareMathOperator*{\dist}{dist}
\DeclareMathOperator*{\supp}{supp}

\DeclareMathOperator*{\gamc}{{\gamma_c}}
\DeclareMathOperator*{\sigc}{{\sigma_c}}

\DeclareMathOperator*{\ima}{Im}
\DeclareMathOperator*{\rea}{Re}
\newcommand{\scal}[1]{\left\langle #1 \right\rangle}

\title[Scattering theory NLS with potential]
{Non-radial scattering theory for nonlinear Schr\"odinger equations with potential}
\author[V. D. Dinh]{Van Duong Dinh}
\address[V. D. Dinh]{Laboratoire Paul Painlev\'e UMR 8524, Universit\'e de Lille CNRS, 59655 Villeneuve d'Ascq Cedex, France
	and 
	Department of Mathematics, Ho Chi Minh City University of Education, 280 An Duong Vuong, Ho Chi Minh City, Vietnam}
\email{contact@duongdinh.com}

\subjclass[2010]{35Q44; 35Q55}
\keywords{Nonlinear Schr\"odinger equation; Kato potential; Scattering; Ground state; Concentration-compactness principle}

\begin{document}
	
	\begin{abstract}
		We consider a class of nonlinear Schr\"odinger equations with potential
		\[
		i\partial_t u +\Delta u - Vu = \pm |u|^\alpha u, \quad (t,x) \in \R \times \R^3,
		\]
		where $\frac{4}{3}<\alpha<4$ and $V$ is a Kato-type potential. We establish a scattering criterion for the equation with non-radial initial data using the ideas of Dodson-Murphy [Math. Res. Lett. 25(6):1805--1825]. As a consequence, we prove the energy scattering for the focusing problem with data below the ground state threshold. Our result extends the recent works of Hong [Commun. Pure Appl. Anal. 15(5):1571--1601] and Hamano-Ikeda [J. Evol. Equ. 2019]. We also study long time dynamics of global solutions to the focusing problem with data at the ground state threshold.
	\end{abstract}

	\maketitle

	\section{Introduction}
	\label{S1}
	\setcounter{equation}{0}
	We consider the Cauchy problem for a class of nonlinear Schr\"odinger equations with potential
	\begin{equation} \label{NLS-V}
	\left\{ 
	\begin{array}{rcl}
	i\partial_t u +\Delta u - Vu &=& \pm |u|^\alpha u, \quad (t,x) \in \R \times \R^3, \\
	u(0,x)&=& u_0(x),
	\end{array}
	\right.
	\end{equation}
	where $u: \mathbb{R} \times \mathbb{R}^3 \rightarrow \mathbb{C}$, $u_0: \mathbb{R}^3 \rightarrow \mathbb{C}$, $\frac{4}{3}<\alpha<4$, and $V$ is a real-valued potential. The range $\frac{4}{3}<\alpha<4$ is referred to the intercritical case which corresponds to the mass-supercritical and energy-subcritical case in three dimensions. The plus (resp. minus) sign in front of the nonlinearity corresponds to the defocusing (resp. focusing) case. In this paper, the potential $V: \R^3 \rightarrow \R$ is assumed to satisfy the following assumptions:
	\begin{align} \label{cond-1-V}
	V \in \mathcal{K} \cap L^{\frac{3}{2}}
	\end{align}
	and
	\begin{align} \label{cond-2-V}
	\|V_-\|_{\mathcal{K}} <4\pi,
	\end{align}
	where $\mathcal{K}$ is a class of Kato potentials with
	\[
	\|V\|_{\mathcal{K}} := \sup_{x\in \R^3} \int_{\R^3} \frac{|V(y)|}{|x-y|} dy
	\]
	and $V_-(x) := \min \{V(x),0\}$ is the negative part of $V$. 
	\begin{remark} \label{rem-exam-V}
		A typical example of potentials satisfying \eqref{cond-1-V} and \eqref{cond-2-V} is the following Yukawa-type potential
		\begin{align} \label{exam-V}
		V(x)= c |x|^{-\sigma} e^{-a|x|}, \quad c \in \R, \quad \sigma \in (0,2), \quad a>0.
		\end{align}
		The genuine Yukawa potential corresponds to $\sigma=1$. Nonlinear Schr\"odinger equation with Yukawa potential appears in a model describing the interaction between a meson field and a fermion field (see e.g., \cite{Yukawa}). We will see in Appendix that
		\begin{align} \label{norm-Lq-V}
		\|V\|_{L^q} = |c| \left[4\pi (aq)^{q\sigma-3} \Gamma(3-q\sigma)\right]^{\frac{1}{q}}
		\end{align}
		and
		\begin{align} \label{norm-K-V}
		\|V\|_{\Kc} = 4\pi |c| a^{2-\sigma} \Gamma(2-\sigma),
		\end{align}
		where $\Gamma$ is the Gamma function.
	\end{remark}
	By the assumptions \eqref{cond-1-V} and \eqref{cond-2-V}, it is known (see e.g., \cite{Hong}) that the operator $\Hc:= -\Delta +V$ has no eigenvalues, and the Schr\"odinger operator $e^{-it\Hc}$ enjoys dispersive and Strichartz estimates. Moreover, the Sobolev norms $\|\Lambda f\|_{L^2}$ and $\|\nabla f\|_{L^2}$ are equivalent, where
	\begin{align} \label{sobo-norm-L2}
	\|\Lambda f\|^2_{L^2} := \int |\nabla f|^2 dx + \int V |f|^2 dx.
	\end{align}
	Thanks to Strichartz estimates, it was shown in \cite{HI, Hong} that the Cauchy problem \eqref{NLS-V} is locally well-posed in $H^1$. In addition, local solutions satisfy the conservation of mass and energy
	\begin{align*}
	M(u(t))&:= \int |u(t,x)|^2 dx = M(u_0), \tag{Mass} \\
	E(u(t))&:= \frac{1}{2} \int |\nabla u(t,x)|^2 + \frac{1}{2} \int V(x) |u(t,x)|^2 dx \pm \frac{1}{\alpha+2} \int |u(t,x)|^{\alpha+2} dx = E(u_0). \tag{Energy}
	\end{align*}
	The main purpose of this paper is to study the energy scattering with non-radial data for \eqref{NLS-V}. 
	\begin{definition} [Energy scattering] A global solution $u \in C(\R,H^1)$ to \eqref{NLS-V} is said to be scattering in $H^1$ forward in time (resp. backward in time) if there exists $u_+\in H^1$ (resp. $u_-\in H^1$) such that
		\[
		\lim_{t\rightarrow +\infty} \|u(t)-e^{-it\Hc} u_+\|_{H^1} =0 \quad \left(\text{resp.} \lim_{t\rightarrow -\infty} \|u(t)-e^{-it\Hc} u_-\|_{H^1}=0 \right).
		\]
	\end{definition} 
	\subsection{Known results}
	Before stating our results, let us recall some known results related to the energy scattering for nonlinear Schr\"odinger equations (NLS) without potential, namely the equation
	\begin{equation} \label{NLS}
	\left\{ 
	\begin{array}{rcl}
	i\partial_t u +\Delta u &=& \pm |u|^\alpha u, \quad (t,x) \in \R \times \R^3, \\
	u(0,x)&=& u_0(x).
	\end{array}
	\right.
	\end{equation}
	It is well-known that \eqref{NLS} is locally well-posed in $H^1$. Moreover, local solutions satisfy the conservation laws of mass and energy
	\begin{align}
	M(u(t))&:= \int |u(t,x)|^2 dx = M(u_0), \nonumber \\
	E_0(u(t))&:=\frac{1}{2} \int |\nabla u(t,x)|^2 dx  \pm \frac{1}{\alpha+2} \int |u(t,x)|^{\alpha+2} dx = E_0(u_0). \label{defi-E0}
	\end{align}
	The equation \eqref{NLS} also satisfies the scaling invariance
	\begin{align} \label{scaling}
	u_\lambda(t,x):=\lambda^{\frac{2}{\alpha}} u(\lambda^2 t, \lambda x), \quad \lambda>0.
	\end{align}
	A direct computation gives
	\[
	\|u_\lambda(0)\|_{\dot{H}^\gamma} =\lambda^{\gamma-\frac{3}{2} +\frac{2}{\alpha}} \|u_0\|_{\dot{H}^\gamma},
	\]
	where $\dot{H}^\gamma$ denotes the homogeneous Sobolev space of order $\gamma$. This shows that the scaling \eqref{scaling} leaves $\dot{H}^{\gamc}$-norm of the initial data invariant, where
	\begin{align} \label{defi-gamc}
	\gamc:=\frac{3}{2}-\frac{2}{\alpha}.
	\end{align}
	We also define the critical exponent
	\begin{align} \label{defi-sigc}
	\sigc:=\frac{1-\gamc}{\gamc}=\frac{4-\alpha}{3\alpha-4}.
	\end{align}
	The energy scattering for \eqref{NLS} in the defocusing case was first established by Ginibre-Velo \cite{GV}. The proof was later simplified by Tao-Visan-Zhang \cite{TVZ}. 
	\begin{theorem}[\cite{GV,TVZ}] Let $\frac{4}{3}<\alpha<4$ and $u_0 \in H^1$. Then the corresponding solution to the defocusing problem \eqref{NLS} exists globally in time and scatters in $H^1$ in both directions.		
	\end{theorem}
	The proof of this result is based on an a priori global bound $\|u\|_{L^4(\R\times \R^3)} \leq C(M,E_0)<\infty$ which is a consequence of interaction Morawetz estimates. We refer the reader to \cite{GV, TVZ} for more details. 

	In the focusing case, it is well-known that \eqref{NLS} admits a global non-scattering solution of the form $u(t,x)=e^{itx} Q(x)$, where 
	$Q$ is the unique positive radial solution to 
	\begin{align} \label{ell-equ}
	-\Delta Q + Q- |Q|^\alpha Q=0.
	\end{align}
	The energy scattering for the focusing problem \eqref{NLS} was first proved by Holmer-Roudenko \cite{HR} with $\alpha=2$ and radially symmetric initial data. The radial assumption was later removed by Duyckaerts-Holmer-Roudenko \cite{DHR}. Extensions of this result to any dimensions $N\geq 1$ and the whole range of the intercritical case were done by Cazenave-Fang-Xie \cite{CFX}, Akahori-Nawa \cite{AN} and Guevara \cite{Guevara}. 
	\begin{theorem}[\cite{AN, CFX, DHR, Guevara, HR}]
		Let $\frac{4}{3}<\alpha<4$. Let $u_0 \in H^1$ satisfy
		\begin{align*}
		E_0(u_0)[M(u_0)]^{\sigc} &< E_0(Q) [M(Q)]^{\sigc}, \\
		\|\nabla u_0\|_{L^2} \|u_0\|^{\sigc}_{L^2} &<\|\nabla Q\|_{L^2} \|Q\|^{\sigc}_{L^2}.
		\end{align*}
		Then the corresponding solution to the focusing problem \eqref{NLS} exists globally in time and scatters in $H^1$ in both directions.
	\end{theorem} 
	The proof of this result is based on the concentration-compactness-rigidity argument of Kenig-Merle \cite{KM}. It consists of three main steps: variational analysis, existence of the minimal blow-up solution via the profile decomposition, and rigidity argument. This method is robust and has been applied to show the energy scattering for various nonlinear Schr\"odinger-type equations. 
	
	Concerning the energy scattering for \eqref{NLS-V}, Hong \cite{Hong} made use of the concentration-compactness-rigidity argument of Colliander-Keel-Staffilani-Takaoka-Tao \cite{CKSTT} and Kenig-Merle \cite{KM} to show the energy scattering for the cubic nonlinearity, i.e. $\alpha=2$. More precisely, he proved the following result.
	\begin{theorem}[\cite{Hong}]  \label{theo-Hong}
		Let $\alpha=2$.
		
		\noindent $\bullet$ (The focusing case) Let $V:\R^3 \rightarrow \R$ satisfy \eqref{cond-1-V}, $V\geq 0$, $x\cdot \nabla V \leq 0$, and $x\cdot \nabla V \in L^{\frac{3}{2}}$. Let $u_0 \in H^1$ satisfy
			\begin{align} 
			E(u_0) M(u_0) &<E_0(Q) M(Q), \label{cond-1-u0-Hong} \\
			\|\Lambda u_0\|_{L^2} \|u_0\|_{L^2} &<\|\nabla Q\|_{L^2} \|Q\|_{L^2}, \label{cond-2-u0-Hong}
			\end{align}
			where $\|\Lambda u_0\|_{L^2}$ is defined as in \eqref{sobo-norm-L2} and $E_0(Q)$ is as in \eqref{defi-E0}. Then the corresponding solution to the focusing problem \eqref{NLS-V} exists globally in time and scatters in $H^1$ in both directions.
			
		\noindent $\bullet$ (The defocusing case) Let $V:\R^3 \rightarrow \R$ satisfy \eqref{cond-1-V}, \eqref{cond-2-V}, and $\|(x\cdot \nabla V)_+\|_{\mathcal{K}} <4\pi$. Let $u_0 \in H^1$. Then the corresponding solution to the defocusing problem \eqref{NLS-V} exists globally in time and scatters in $H^1$ in both directions.
	\end{theorem}
	The proof of this result depends heavily on linear profile decomposition. However, due to the lack of translation invariance for both linear and nonlinear equations caused by the potential, showing the linear profile decomposition is more involved. To overcome the difficulty, Hong considered the potential as a perturbation of the linear equation, and chose a suitable Strichartz norm to make the error small. We refer the reader to \cite{Hong} for more details.

	Recently, Hamano-Ikeda \cite{HI} extended Hong's result to the whole range of the intercritical case, i.e. $\frac{4}{3}<\alpha<4$ and radially symmetric initial data. More precisely, they proved the following result.
	\begin{theorem}[\cite{HI}] \label{theo-HI}
		Let $\frac{4}{3}<\alpha<4$. Let $V:\R^3 \rightarrow \R$ be radially symmetric satisfying \eqref{cond-1-V}, $V\geq 0$, $x\cdot \nabla V \leq 0$, and $x \cdot \nabla V \in L^{\frac{3}{2}}$. Let $u_0 \in H^1$ be radially symmetric satisfying
		\begin{align} 
		E(u_0) [M(u_0)]^{\sigc} &< E_0(Q) [M(Q)]^{\sigc}, \label{cond-1-HI} \\
		\|\nabla u_0\|_{L^2} \|u_0\|^{\sigc}_{L^2} &< \|\nabla Q\|_{L^2} \|Q\|^{\sigc}_{L^2}. \label{cond-2-HI}
		\end{align}
		Then the corresponding solution to the focusing problem \eqref{NLS-V} exists globally in time and scatters in $H^1$ in both directions.
	\end{theorem}
	The proof of this result is based on a recent argument of Dodson-Murphy \cite{DM} which makes use of the radial assumption. This is done by three main steps. The first one is to use nonlinear estimates to show a suitable scattering criterion. The second one is to use variational arguments to derive the coercivity on sufficiently large balls. Finally, thanks to the coercivity, Morawetz estimates, and the radial Sobolev embedding, one obtains a space time decay which implies the smallness of $L^2$-norm of the solution inside a large ball for sufficiently large time. This together with dispersive estimates imply that the global solution satisfies the scattering criterion.
	
	In the defocusing case, the energy scattering for non-radial data was proved by the first author in \cite[Theorem 1.4]{Dinh-INLS-poten}. 
	\begin{theorem} [\cite{Dinh-INLS-poten}]
		Let $\frac{4}{3}<\alpha<4$. Let $V:\R^3 \rightarrow \R$ be radially symmetric satisfying \eqref{cond-1-V}, \eqref{cond-2-V}, $x\cdot \nabla V \leq 0$, and $\partial_r V \in L^q$ for any $\frac{3}{2}\leq q\leq \infty$. Let $u_0 \in H^1$. Then the corresponding solution to the defocusing problem \eqref{NLS-V} exists globally in time and scatters in $H^1$ in both directions.
	\end{theorem}
	The proof of this result is based on interaction Morawetz estimates in the same spirit of \cite{TVZ}. The first step is to use interaction Morawetz estimates to show a priori global bound $\|u\|_{L^4(\R\times \R^3)} \leq C(M,E)<\infty$. This global bound combined with nonlinear estimates show the energy scattering. For more details, we prefer the reader to \cite[Appendix]{Dinh-INLS-poten}.
	
	The energy scattering for NLS with other-type of potentials have been studied in other works. For instance, Carles \cite{Carles} proved the energy scattering for a smooth real-valued potential satisfying for $\mu>2$,
	\[
	|\partial^\alpha V(x)| \leq \frac{C_\alpha}{(1+|x|)^{\mu+|\alpha|}}, \quad \forall \alpha \in \R^N
	\]
	and there exists $M=M(N, \mu)>0$ such that 
	\[
	\left(\frac{x}{|x|} \cdot \nabla V(x)\right)_+ \leq \frac{M}{(1+|x|)^{\mu+1}}, \quad \forall x\in \R^N,
	\]
	where $f_+:= \max \{0,f\}$.	Lafontaine \cite{Lafontaine} proved the energy scattering for 1D NLS with non-negative potential satisfying
	\[
	V\in L^1_1(\R), \quad V' \in L^1_1(\R), \quad x V' \leq 0,
	\]
	where
	\[
	\|V\|_{L^1_1(\R)}:= \int_{\R} |V(x)| (1+|x|) dx.
	\]
	We also refer to other related works of Banica-Visciglia \cite{BV}, Killip-Murphy-Visan-Zheng \cite{KMVZ}, Lu-Miao-Murphy \cite{LMM}, Zheng \cite{Zheng} and Forcella-Visciglia \cite{FV}. 
	
	\subsection{Main result}
	Inspiring by the aforementioned results, the purpose of this paper is to show the energy scattering for \eqref{NLS-V} with non-radially symmetric initial data. More precisely, we prove the following scattering criterion.
	
	\begin{theorem}[Scattering criterion] \label{theo-scat-crite}
		Let $\frac{4}{3}<\alpha<4$. 
		
		\noindent $\bullet$ (The focusing case) Let $V:\R^3 \rightarrow \R$ be radially symmetric satisfying \eqref{cond-1-V}, $V\geq 0$, $x\cdot \nabla V \leq 0$, and $\partial_r V \in L^q$ for any $\frac{3}{2}\leq q\leq \infty$. Let $u$ be a $H^1$-solution to the focusing problem \eqref{NLS-V} defined on the maximal forward time interval of existence $[0,T^*)$. Assume that 
			\begin{align} \label{scat-cond}
			\sup_{t\in [0,T^*)} \|u(t)\|^{\alpha+2}_{L^{\alpha+2}} \|u(t)\|^{2\sigc}_{L^2} < \|Q\|^{\alpha+2}_{L^{\alpha+2}} \|Q\|^{2\sigc}_{L^2}.
			\end{align}
			Then the solution exists globally in time, i.e. $T^*<\infty$, and scatters in $H^1$ forward in time. A similar result holds for the negative times. 
						
		\noindent $\bullet$ (The defocusing case) Let $V:\R^3 \rightarrow \R$ satisfy \eqref{cond-1-V}, \eqref{cond-2-V}, and 
			\begin{align} \label{cond-3-V}
			\left\{
			\begin{array}{l}
			\text{either } V \text{ be radially symmetric}, x \cdot \nabla V \leq 0, \text{ and } \partial_r V \in L^q \text{ for any } \frac{3}{2} \leq q \leq \infty;\\
			\text{or } V \text{ be non-radially symmetric}, x \cdot \nabla V \in L^{\frac{3}{2}}, x \cdot \nabla V \leq 0, \text{ and } \nabla^2 V \text{ be non-positive definite.}
			\end{array}
			\right.
			\end{align}
			Let $u_0 \in H^1$. Then the corresponding solution to the defocusing problem \eqref{NLS} exists globally in time and scatters in $H^1$ in both directions.
	\end{theorem}
	
	Theorem \ref{theo-scat-crite} gives a general criterion for the energy scattering for \eqref{NLS-V}. In the focusing case, Theorem \ref{theo-scat-crite} allows us to study long time dynamics of solutions with data lying both below and at the ground state threshold (see Theorem \ref{theo-scat-below} and Theorem \ref{theo-scat-at}). In the defocusing case, Theorem \ref{theo-scat-crite} not only gives an alternative proof for the energy scattering given in \cite[Theorem 1.4]{Dinh-INLS-poten} but also extends this result to the case of non-radially symmetric potential. Moreover, comparing to \cite{Hong}, we do not assume any smallness condition on $\|(x \cdot \nabla V)_+\|_{\Kc}$.
	
	\begin{remark} 
		The condition $\partial_r V \in L^q$ for any $\frac{3}{2} \leq q \leq \infty$ is needed to ensure $\partial_r V |u(t)|^2  \in L^1$ (see Remark \ref{rem-Lq}). One may relax this assumption to $\partial_r V \in L^q + L^\infty$ for some $q \geq \frac{3}{2}$.
	\end{remark}
	
	\begin{remark} \label{rem-zero-poten}
		There is no non-zero potential $V$ satisfying the following properties: $V$ is radially symmetric, $V \in L^{\frac{3}{2}}$, $V\geq 0$, $x \cdot \nabla V \in L^{\frac{3}{2}}$, $x\cdot \nabla V \leq 0$, and $\nabla^2 V$ is non-positive definite. Under these assumptions, $V$ is non-negative, concave, and decreasing in the radial direction. This potential does not belong to $L^{\frac{3}{2}}$ except $V \equiv 0$.
	\end{remark}
	
	\begin{remark}
		It was pointed out in \cite{HI} using the result of \cite{Mizutani-PAMS} that if $V\in L^{\frac{3}{2}}$ and $V\geq 0$, then there exist $f_{\pm} \in H^1$ such that
		\[
		\lim_{t\rightarrow \pm \infty} \|e^{-it\Hc} f - e^{it\Delta} f_{\pm}\|_{H^1} =0.
		\]
		By this result, the scattering for the focusing case given in Theorem $\ref{theo-scat-below}$ can be rewritten as: there exist $u_+ \in H^1$ such that 
		\[
		\lim_{t\rightarrow + \infty} \|u(t)-e^{it\Delta} u_+ \|_{H^1} =0, 
		\]
		i.e. the solution behaves like the linear solution without potential at infinity.
	\end{remark}

	A first application of Theorem \ref{theo-scat-crite} is the following energy scattering below the ground state threshold.
	
	\begin{theorem} [Scattering below the ground state threshold] \label{theo-scat-below}
		Let $\frac{4}{3}<\alpha<4$. Let $V:\R^3 \rightarrow \R$ be radially symmetric satisfying \eqref{cond-1-V}, $V\geq 0$, $x\cdot \nabla V \leq 0$, and $\partial_r V \in L^q$ for any $\frac{3}{2} \leq q \leq \infty$. Let $u_0 \in H^1$ satisfy \eqref{cond-1-HI} and \eqref{cond-2-HI}. Then the corresponding solution to the focusing problem \eqref{NLS-V} satisfies 
		\begin{align} \label{est-solu-below}
		\sup_{t\in (-T_*,T^*)} \|u(t)\|^{\alpha+2}_{L^{\alpha+2}} \|u(t)\|^{2\sigc}_{L^2} < \|Q\|^{\alpha+2}_{L^{\alpha+2}}\|Q\|^{2\sigc}_{L^2},
		\end{align}
		where $(-T_*,T^*)$ is the maximal time interval of existence. In particular, the solution exists globally in time, and scatters in $H^1$ in both directions.			
	\end{theorem}

	\begin{remark}
		Comparing to \cite{Hong}, our result extends the one in \cite{Hong} (with radially symmetric potential) to the whole range of the intercritical case. Comparing to \cite{HI}, our result improves the one in \cite{HI} by removing the radial assumption on initial data. 
	\end{remark}
	
	Another application of Theorem \ref{theo-scat-crite} is the following long time dynamics for solutions lying at the ground state threshold for the focusing problem \eqref{NLS-V}.
	
	\begin{theorem} [Scattering at the ground state threshold] \label{theo-scat-at}
		Let $\frac{4}{3}<\alpha<4$. Let $V:\R^3 \rightarrow \R$ be radially symmetric satisfying \eqref{cond-1-V}, $V\geq 0$, $x\cdot \nabla V \leq 0$, and $\partial_r V \in L^q$ for any $\frac{3}{2} \leq q \leq \infty$. Let $u_0 \in H^1$ satisfy
		\begin{align} \label{cond-enegy-at}
		E(u_0) [M(u_0)]^{\sigc} = E_0(Q) [M(Q)]^{\sigc}
		\end{align}
		and
		\begin{align} \label{cond-grad-at}
		\|\nabla u_0\|_{L^2} \|u_0\|^{\sigc}_{L^2} < \|\nabla Q\|_{L^2} \|Q\|^{\sigc}_{L^2}.
		\end{align}
		Then the corresponding solution to the focusing problem \eqref{NLS-V} exists globally in time. Moreover, the solution either scatters in $H^1$ forward in time, or there exist a time sequence $t_n\rightarrow \infty$ and a sequence $(y_n)_{n\geq 1} \subset \R^3$ satisfying $|y_n| \rightarrow \infty$ such that
		\[
		u(t_n, \cdot+y_n) \rightarrow e^{i\theta} \lambda Q \quad \text{strongly in } H^1 
		\]
		for some $\theta \in \R$ and $\lambda:= \frac{\|u_0\|_{L^2}}{\|Q\|_{L^2}}$ as $n\rightarrow \infty$.
	\end{theorem}
	
	To our knowledge, the first result studied long time dynamics of solutions to the focusing nonlinear Schr\"odinger equation with data at the ground state threshold belongs to Duyckaerts-Roudenko \cite{DR}. They have showed qualitative properties of solutions at the ground state threshold based on delicate spectral estimates. However, their results are limited to the case of cubic nonlinearity, i.e. $\alpha=2$ in \eqref{NLS}. Recently, the first author in \cite{Dinh-DCDS} gave a simple proof for long time dynamics of solutions to the focusing  NLS with data at the ground state threshold in any dimensions. Our result is an extension of that in \cite{Dinh-DCDS} to the case of external potential. 
	
	\subsection{Outline of the proof} \label{subsec:outline}
	The proof of Theorem $\ref{theo-scat-below}$ is based on recent arguments of Dodson-Murphy \cite{DM-nonrad} which do not use the concentration-compactness-rigidity argument. The proof makes use of a suitable scattering criterion and the interaction Morawetz estimate as follows. First, by using Strichartz estimates, it was shown in \cite{HI} that if $u$ is a global solution to \eqref{NLS} satisfying 
	\[
	\sup_{t\in \R} \|u(t)\|_{H^1} \leq A
	\]
	for some constant $A>0$, then there exists $\delta=\delta(A)>0$ such that if 
	\begin{align} \label{scat-1-intro}
	\|e^{i(t-T)\Delta} u(T)\|_{L^q([T,\infty)\times \R^3)} <\delta
	\end{align}
	for some $T>0$, where $	q:=\frac{5\alpha}{2}$, then the solution scatters in $H^1$ forward in time. Second, thanks to dispersive estimates, the condition \eqref{scat-1-intro} is later reduced to show that there exist $\vareps=\vareps(A)>0$ and $T_0=T_0(\vareps,A)>0$ such that for any $a\in \R$, there exists $t_0 \in (a,a+T_0)$ such that $[t_0-\vareps^{-\sigma},t_0] \subset (a,a+T_0)$ and
	\begin{align} \label{scat-2-intro}
	\|u\|_{L^q([t_0-\vareps^{-\sigma},t_0]\times \R^3)} \lesssim \vareps^\mu
	\end{align}
	for some $\sigma, \mu>0$ satisfying
	\begin{align} \label{cond-sigma-mu-intro}
	\mu \alpha -\frac{\sigma}{10}>0.
	\end{align} 
	Third, to show \eqref{scat-2-intro}, we rely on the interaction Morawetz estimate introduced by Dodson-Murphy \cite{DM-nonrad}. More precisely, we consider the interaction Morawetz action
	\[
	\Mca^{\otimes 2}_R(t):= \iint |u(t,y)|^2 \psi(x-y)(x-y)\cdot 2 \ima(\overline{u}(t,x) \nabla u(t,x)) dxdy,
	\]
	where $\psi$ is a suitable localization. Taking into account the coercivity property of solutions and using the Galelian transformation, we show that there exists $T_0=T_0(\vareps), J=J(\vareps)$, $R_0=R_0(u_0,Q)$ such that for any $a \in \R$,
	\begin{align} \label{est-inte-mora-intro}
	\frac{1}{JT_0} \int_a^{a+T_0} \int_{R_0}^{R_0e^J} \frac{1}{R^N} \iiint |\chi_R(y-z) u(t,y)|^2 |\nabla[\chi_R(x-z) u^\xi(t,x)]|^2 dxdydz \frac{dR}{R} dt \lesssim \vareps,
	\end{align}
	where $\chi_R(x)=\chi(x/R)$ is a cutoff function and $u^\xi(t,x) =e^{ix\cdot \xi} u(t,x)$ with some $\xi = \xi(t,z,R) \in \R^3$. In the case $V$ is radially symmetric, we also make use of an estimate related to the Morawetz action
	\[
	\Mca_R(t):= \int \psi(x)x\cdot 2 \ima(\overline{u}(t,x) \nabla u(t,x)) dx.
	\]
	Finally, thanks to \eqref{est-inte-mora-intro}, an orthogonal argument similar to that of \cite{XZZ} implies \eqref{scat-2-intro}. 
	
	This paper is organized as follows. In Section $\ref{S2}$, we recall some preliminaries including dispersive estimates, Strichartz estimates, and the equivalence of Sobolev norms. In Section $\ref{S3}$, we recall the local well-posedness and show a suitable scattering criterion for \eqref{NLS-V}. Section $\ref{S4}$ is devoted to the proof of the interaction Morawetz estimate. The proof of Theorem $\ref{theo-scat-crite}$ is given in Section $\ref{S5}$. We give the proofs of Theorem \ref{theo-scat-below} and Theorem \ref{theo-scat-at} in Section \ref{S6}. Finally, a remark on long time dynamics for nonlinear Schr\"odinger equations with repulsive inverse-power potentials is given in Section \ref{S7}.
	
	\section{Preliminaries}
	\label{S2}
	\setcounter{equation}{0}
	In this section, we recall some useful estimates related to the Schr\"odinger operator with Kato potentials.
	
	\subsection{Dispersive estimate}
	\begin{lemma} [Dispersive estimate \cite{Hong}] Let $V:\R^3 \rightarrow \R$ satisfy \eqref{cond-1-V} and \eqref{cond-2-V}. Then we have
		\begin{align} \label{dis-est}
		\|e^{-it \Hc} f\|_{L^\infty} \lesssim |t|^{-\frac{3}{2}} \|f\|_{L^1}
		\end{align}
		for any $f \in L^1$.
	\end{lemma}
	
	\subsection{Strichartz estimates}
	Let $I \subset \R$ be an interval and $q,r \in [1,\infty]$. We define the mixed norm
	\[
	\|u\|_{L^q(I,L^r)} := \left(\int_I \left( \int_{\R^3} |u(t,x)|^r dx \right)^{\frac{q}{r}} dt \right)^{\frac{1}{q}}
	\]
	with a usual modification when either $q$ or $r$ are infinity. When $q=r$, we use the notation $L^q(I \times \R^3)$ instead of $L^q(I,L^q)$.
	
	\begin{definition}
		A pair $(q,r)$ is said to be Schr\"odinger admissible, for short $(q,r)\in S$, if
		\[
		\frac{2}{q}+\frac{3}{r}=\frac{3}{2}, \quad r\in [2,6].
		\]
	\end{definition}
	
	Thanks to dispersive estimates \eqref{dis-est}, the abstract theory of Keel-Tao \cite{KT} implies the following Strichartz estimates.
	\begin{proposition}[Strichartz estimates \cite{Hong}]
		Let $V:\R^3 \rightarrow \R$ satisfy \eqref{cond-1-V} and \eqref{cond-2-V}. Let $I \subset \R$ be an interval. Then there exists a constant $C>0$ independent of $I$ such that the following estimates hold:
		\begin{itemize}
			\item (Homogeneous estimates)
			\[
			\|e^{-it\Hc} f\|_{L^q(I,L^r)} \leq C \|f\|_{L^2}
			\]
			for any $f \in L^2$ and any Schr\"odinger admissible pair $(q,r)$.
			\item (Inhomogeneous estimates)
			\[
			\left\|\int_0^t e^{-i(t-s)\Hc} F(s) ds \right\|_{L^q(I,L^r)} \leq C \|F\|_{L^{m'}(I,L^{n'})}
			\]
			for any $F \in L^{m'}(I,L^{n'})$ and any Schr\"odinger admissible pairs $(q,r), (m,n)$, where $(m,m')$ and $(n,n')$ are H\"older conjugate pairs.
		\end{itemize}
	\end{proposition}

	\subsection{The equivalence of Sobolev norms}
	Let $\gamma \geq 0$. We define the homogeneous and inhomogeneous Sobolev spaces associated to $\Hc$ as the closure of $C^\infty_0(\R^3)$ under the norms
	\[
	\|f\|_{\dot{W}^{\gamma,r}_V}:=\|\Lambda^\gamma f \|_{L^r}, \quad \|f\|_{W^{\gamma,r}_V}:= \|\scal{\Lambda}^\gamma f\|_{L^r}, \quad \Lambda:=\sqrt{\Hc}, \quad \scal{\Lambda}:=\sqrt{1+\Hc}.
	\]
	When $r=2$, we abbreviate $\dot{H}^\gamma_V:= \dot{W}^{\gamma,2}_V$ and $H^\gamma_V:=W^{\gamma,2}_V$. We have the following Sobolev estimates and the equivalence of Sobolev spaces due to Hong \cite{Hong}. 
	\begin{lemma} [Sobolev estimates \cite{Hong}]
		Let $V:\R^3 \rightarrow \R$ satisfy \eqref{cond-1-V} and \eqref{cond-2-V}. Then we have
		\[
		\|f\|_{L^q} \lesssim \|f\|_{\dot{W}^{\gamma,r}_V}, \quad \|f\|_{L^q} \lesssim \|f\|_{W^{\gamma,r}_V},
		\]
		where $1<r<q<\infty$, $1<r<\frac{3}{\gamma}$, $0\geq \gamma\leq 2$ and $\frac{1}{q}=\frac{1}{r}-\frac{\gamma}{3}$.
	\end{lemma}
	
	\begin{lemma} [Equivalence of Sobolev spaces \cite{Hong}] Let $V:\R^3 \rightarrow \R$ satisfy \eqref{cond-1-V} and \eqref{cond-2-V}. Then we have 
		\[
		\|f\|_{\dot{W}^{\gamma,r}_V} \sim \|f\|_{\dot{W}^{\gamma,r}}, \quad \|f\|_{W^{\gamma,r}_V} \sim \|f\|_{W^{\gamma,r}},
		\]
		where $1<r<\frac{3}{\gamma}$ and $0\leq \gamma \leq 2$.
	\end{lemma}

	\section{Local theory}
	\label{S3}
	\setcounter{equation}{0}
	
	\subsection{Local well-posedness}
	We recall the following local well-posedness and small data scattering for \eqref{NLS-V} due to Hamano-Ikeda \cite{HI}.
	
	\begin{lemma} [Local well-posedness \cite{HI}]
		Let $0<\alpha<4$. Let $V:\R^3 \rightarrow \R$ satisfy \eqref{cond-1-V} and \eqref{cond-2-V}. Let $u_0 \in H^1$. Then there exists $T=T(\|u_0\|_{H^1})>0$ and a unique solution 
		\[
		u \in C([-T,T],H^1) \cap L^q([-T,T],W^{1,r}_V)
		\]
		to \eqref{NLS-V} for any Schr\"odinger admissible pair $(q,r)$.
	\end{lemma}
	
	\begin{lemma} [Small data scattering \cite{HI}] \label{lem-small-scat}
		Let $\frac{4}{3}<\alpha<4$. Let $V: \R^3 \rightarrow \R$ satisfy \eqref{cond-1-V} and \eqref{cond-2-V}. Suppose $u$ is a global solution to \eqref{NLS-V} satisfying 
		\[
		\|u\|_{L^\infty(\R, H^1)} \leq A
		\]
		for some constant $A>0$. Then there exists $\delta=\delta(A)>0$ such that if
		\[
		\|e^{-i(t-T)\Hc} u(T)\|_{L^q([T,\infty)\times \R^3)} <\delta
		\]
		for some $T>0$, where 
		\begin{align} \label{defi-q}
		q:=\frac{5\alpha}{2}, 
		\end{align}
		then $u$ scatters in $H^1$ forward in time. 
	\end{lemma}
	We refer the reader to \cite[Lemma 4.2 and Lemma 4.3]{HI} for the proof of the above results.
	
	\subsection{Scattering criteria}
	
	\begin{lemma} [Scattering criteria] \label{lem-scat-crite}
		Let $\frac{4}{3}<\alpha<4$. Let $V:\R^3 \rightarrow \R$ satisfy \eqref{cond-1-V} and \eqref{cond-2-V}. Suppose that $u$ is a global solution to \eqref{NLS-V} satisfying
		\[
		\|u\|_{L^\infty(\R,H^1)} \leq A
		\]
		for some constant $A>0$. Then there exist $\vareps=\vareps(A)>0$ sufficiently small and $T_0=T_0(\vareps,A)>0$ sufficiently large such that if for any $a \in \R$, there exists $t_0 \in (a,a+T_0)$ such that $[t_0-\vareps^{-\sigma},t_0] \subset (a,a+T_0)$ and
		\begin{align} \label{scat-cond-a}
		\|u\|_{L^q([t_0-\vareps^{-\sigma},t_0] \times \R^3)} \lesssim \vareps^{\mu}
		\end{align}
		for some $\sigma, \mu>0$ satisfying 
		\begin{align} \label{cond-sigma-mu}
		\mu \alpha -\frac{\sigma}{10}>0, 
		\end{align}
		where $q$ is as in \eqref{defi-q}, then $u$ scatters in $H^1$ forward in time.
	\end{lemma}
	
	\begin{proof}
		By Lemma $\ref{lem-small-scat}$, it suffices to show that there exists $T>0$ such that 
		\begin{align} \label{small-cond-T}
		\|e^{-i(t-T)\Hc} u(T)\|_{L^q([T,\infty)\times \R^3)} \lesssim \vareps^\vartheta
		\end{align}
		for some $\vartheta>0$.
		
		To show \eqref{small-cond-T}, we first write
		\[
		e^{-i(t-T)\Hc} u(T)=e^{-it\Hc}u_0 \mp i \int_0^T e^{-i(t-s)\Hc} |u(s)|^\alpha u(s) ds.
		\]
		By Sobolev embedding and Strichartz estimates, we have
		\[
		\|e^{-it\Hc} u_0\|_{L^q(\R\times \R^3)} \lesssim \|\Lambda^{\gamc} e^{-it\Hc} u_0\|_{L^q(\R,L^r)} \lesssim \|\Lambda^{\gamc} u_0\|_{L^2} \lesssim \|u_0\|_{H^1}<\infty,
		\]
		where
		\begin{align} \label{defi-r}
		r:=\frac{30\alpha}{15\alpha-8}
		\end{align}
		is so that $(q,r)\in S$. By the monotone convergence theorem, there exists $T_1>0$ sufficiently large such that for any $T>T_1$,
		\begin{align} \label{est-line-part}
		\|e^{-it\Hc} u_0\|_{L^q([T,\infty)\times \R^3)} \lesssim \vareps.
		\end{align}
		Taking $a=T_1$ and $T=t_0$ with $a$ and $t_0$ as in \eqref{scat-cond-a}, we write
		\begin{align*}
		i \int_0^T e^{-i(t-s)\Hc} |u(s)|^\alpha u(s) ds &= i \int_I e^{-i(t-s)\Hc} |u(s)|^\alpha u(s) ds + i\int_J e^{-i(t-s)\Hc} |u(s)|^\alpha u(s) ds \\
		&=:F_1 (t) + F_2(t),
		\end{align*}
		where $I:= [0,T-\vareps^{-\sigma}]$ and $J:= [T-\vareps^{-\sigma},T]$. 
		
		By Sobolev embedding, Strichartz estimates, \eqref{scat-cond-a} and \eqref{cond-sigma-mu}, we see that
		\begin{align} \label{est-F2}
		\|F_2\|_{L^q([T,+\infty)\times \R^3)} &\lesssim \|\Lambda^{\gamc}(|u|^\alpha u)\|_{L^2(J,L^{\frac{6}{5}})} \nonumber \\
		&\lesssim \||\nabla|^{\gamc}(|u|^\alpha u)\|_{L^2(J,L^{\frac{6}{5}})} \nonumber \\
		&\lesssim \|u\|^\alpha_{L^q(J\times \R^3)} \||\nabla|^{\gamc} u\|_{L^{10}(J,L^{\frac{30}{13}})} \nonumber \\
		&\lesssim \vareps^{\mu\alpha-\frac{\sigma}{10}}.
		\end{align}
		Here we have used 
		\[
		\||\nabla|^{\gamc} u\|_{L^{10}(J,L^{\frac{30}{13}})} \sim \|\Lambda^{\gamc} u\|_{L^{10}(J,L^{\frac{30}{13}})} \lesssim \scal{J}^{\frac{1}{10}}
		\]
		which follows from the local well-posedness and the fact that $\left(10,\frac{30}{13}\right) \in S$.
		
		We next estimate $F_1$. By H\"older's inequality, we have
		\[
		\|F_1\|_{L^q([T,+\infty) \times \R^3)} \leq \|F_1\|^{\theta}_{L^k([T,+\infty),L^l)} \|F_1\|^{1-\theta}_{L^p([T,+\infty),L^\infty)}
		\]
		where $\theta \in (0,1)$ and 
		\begin{align} \label{cond-klp}
		\frac{1}{q}=\frac{\theta}{k}+\frac{1-\theta}{p}=\frac{\theta}{l}
		\end{align}
		for some $k,l$ and $p$ to be chosen later. We first choose $k$ and $l$ so that $(k,l) \in S$. Then, using the fact that
		\[
		F_1(t) = e^{-i(t-T+\vareps^{-\sigma})\Hc} u(T-\vareps^{-\sigma}) - e^{-it\Hc} u_0,
		\] 
		we have
		\[
		\|F_1\|_{L^k([T,+\infty),L^l)} \lesssim 1.
		\]
		We next estimate, by dispersive estimates \eqref{dis-est}, that for $t \in [T,\infty)$,
		\begin{align*}
		\|F_1(t)\|_{L^\infty} &\lesssim \int_0^{T-\vareps^{-\sigma}} \|e^{-i(t-s)\Hc} |u(s)|^\alpha u(s)\|_{L^\infty} ds \\
		&\lesssim \int_0^{T-\vareps^{-\sigma}} (t-s)^{-\frac{3}{2}} \|u(s)\|^{\alpha+1}_{L^{\alpha+1}} ds \\
		&\lesssim (t-T+\vareps^{-\sigma})^{-\frac{1}{2}}.
		\end{align*}
		It follows that
		\begin{align*}
		\|F_1\|_{L^p([T,\infty),L^\infty)} &\lesssim \left( \int_T^\infty (t-T+\vareps^{-\sigma})^{-\frac{p}{2}} dt\right)^{\frac{1}{p}} \\
		&\lesssim \vareps^{\sigma \left(\frac{1}{2}-\frac{1}{p}\right)}
		\end{align*}
		provided that $p>2$. We thus get
		\begin{align} \label{est-F1}
		\|F_1\|_{L^q([T,\infty)\times \R^3)} \lesssim \vareps^{\sigma \left(\frac{1}{2}-\frac{1}{p}\right) (1-\theta)}.
		\end{align}
		
		We will choose $\theta \in (0,1)$ and $k,l, p$ satisfying \eqref{cond-klp}, $(k, l) \in S$ and $p>2$. By \eqref{cond-klp}, we have
		\[
		l = \theta q= \frac{5\alpha \theta}{2}.
		\]
		To make $(k,l) \in S$, we need $l \in [2,6]$ which implies $\theta \in \left[\frac{4}{5\alpha}, \frac{12}{5\alpha}\right]$. We also have
		\[
		k=\frac{20\alpha\theta}{15\alpha\theta-12}, \quad p=\frac{20\alpha(1-\theta)}{20-15\alpha \theta}.
		\]
		Note that $p>2$ is equivalent to $\theta>\frac{4-2\alpha}{\alpha}$. In the case $2\leq \alpha<4$, we can choose $\theta =\frac{4}{5\alpha}$. In the case $\frac{4}{3}<\alpha<2$, we can choose $\theta = \max \left\{\frac{4}{5\alpha}, \frac{4-2\alpha}{\alpha}+\right\}$.
				
		Collecting \eqref{est-line-part}, \eqref{est-F2} and \eqref{est-F1}, we prove \eqref{small-cond-T}. The proof is complete.	
	\end{proof}
	
	\section{Interaction Morawetz estimates}
	\label{S4}
	\setcounter{equation}{0}
	\subsection{Variational analysis}
	We recall some properties of the ground state $Q$ which is the unique positive radial solution to \eqref{ell-equ}. The ground state $Q$ optimizes the sharp Gagliardo-Nirenberg inequality
	\begin{align} \label{GN-ineq}
	\|f\|_{L^{\alpha+2}}^{\alpha+2} \leq C_{\opt} \|\nabla f\|_{L^2}^{\frac{3\alpha}{2}} \|f\|_{L^2}^{\frac{4-\alpha}{2}}, \quad f \in H^1(\R^3),
	\end{align}
	that is
	\[
	C_{\opt} = \|Q\|^{\alpha+2}_{L^{\alpha+2}} \div \left[\|\nabla Q\|_{L^2}^{\frac{3\alpha}{2}} \|Q\|^{\frac{4-\alpha}{2}}_{L^2} \right].
	\]
	Using the following Pohozaev's identities (see e.g., \cite{Cazenave})
	\begin{align} \label{poho-iden}
	\|Q\|^2_{L^2} = \frac{4-\alpha}{3\alpha} \|\nabla Q\|^2_{L^2} = \frac{4-\alpha}{2(\alpha+2)} \|Q\|^{\alpha+2}_{L^{\alpha+2}},
	\end{align}
	we infer that
	\begin{align} \label{sharp-const-GN}
	C_{\opt} = \frac{2(\alpha+2)}{3\alpha} \left( \|\nabla Q\|_{L^2} \|Q\|^{\sigc}_{L^2}\right)^{-\frac{3\alpha-4}{2}}.
	\end{align}
	Moreover,
	\begin{align} \label{energy-Q}
	E_0(Q) = \frac{3\alpha-4}{6\alpha} \|\nabla Q\|^2_{L^2} = \frac{3\alpha-4}{4(\alpha+2)} \|Q\|^{\alpha+2}_{L^{\alpha+2}},
	\end{align}
	where $E_0(Q)$ is as in \eqref{defi-E0}. In particular,
	\begin{align} \label{iden-Q}
	E_0(Q)[M(Q)]^{\sigc} = \frac{3\alpha-4}{6\alpha} \left( \|\nabla Q\|_{L^2} \|Q\|^{\sigc}_{L^2} \right)^2.
	\end{align}
	
	We also have the following refined Gagliardo-Nirenberg inequality due to Dodson-Murphy \cite{DM-nonrad}.
	\begin{lemma} [\cite{DM-nonrad}] \label{lem-GN}
		Let $0<\alpha<4$. Then for any $f \in H^1$ and any $\xi \in \R^3$,
		\begin{align} \label{refi-GN}
		\|f\|^{\alpha+2}_{L^{\alpha+2}} \leq \frac{2(\alpha+2)}{3\alpha} \left(\frac{\|\nabla f\|_{L^2} \|f\|^{\sigc}_{L^2}}{\|\nabla Q\|_{L^2} \|Q\|^{\sigc}_{L^2}} \right)^{\frac{3\alpha-4}{2}} \|\nabla[e^{ix\cdot \xi} f]\|^2_{L^2}.
		\end{align}
	\end{lemma}

	\subsection{Interaction Morawetz estimate}
	Let $\eta \in (0,1)$ be a small constant, and $\chi$ be a smooth decreasing radial function satisfying
	\begin{align} \label{defi-chi}
	\chi(x) = \left\{
	\begin{array}{ccl}
	1 &\text{if}& |x| \leq 1-\eta, \\
	0 &\text{if}& |x|>1.
	\end{array}
	\right.
	\end{align}
	For $R>0$ large, we define the functions
	\begin{align} \label{defi-phi}
	\phi_R(x):= \frac{1}{\omega_3 R^3} \int \chi^2_R(x-z) \chi^2_R(z) dz 
	\end{align}
	and
	\begin{align} \label{defi-phi-1}
	\phi_{1,R}(x):= \frac{1}{\omega_3 R^3} \int \chi^2_R(x-z) \chi^{\alpha+2} _R(z) dz
	\end{align}
	where $\chi_R(z):=\chi(z/R)$, and $\omega_3$ is the volume of the unit ball in $\R^3$. We see that $\phi_R$ and $\phi_{1,R}$ are radial functions. We next define the radial function
	\begin{align} \label{defi-psi}
	\psi_R(x) = \psi_R(r):= \frac{1}{r} \int_0^{r} \phi_R(\tau) d\tau, \quad r=|x|.
	\end{align}
	We collect some properties of $\phi_R$ and $\psi_R$ as follows.
	\begin{lemma} [\cite{DM-nonrad}]\label{lem-proper-psi}
		We have 
		\begin{align} \label{proper-psi-1}
		|\psi_R(x)| \lesssim \min \left\{1,\frac{R}{|x|} \right\},  \quad \partial_j \psi_R (x) =\frac{x_j}{|x|^2} \left(\phi_R(x) - \psi_R(x)\right), \quad j=1, \cdots, 3
		\end{align}
		and
		\begin{align} \label{proper-psi-2}
		\psi_R(x) - \phi_R(x) \geq 0, \quad |\nabla \phi_R(x)| \lesssim \frac{1}{R}, \quad
		|\phi_R(x) - \phi_{1,R}(x)| \lesssim \eta
		\end{align}
		and
		\begin{align} \label{proper-psi-3}
		\psi_R(x)|x| \sim R, \quad |\psi_R(x) - \phi_R(x)| \lesssim \min \left\{ \frac{|x|}{R}, \frac{R}{|x|}\right\}, \quad |\nabla \psi_R(x)| \lesssim \min \left\{\frac{1}{R}, \frac{R}{|x|^2} \right\}
		\end{align}
		for all $x \in \R^3$.
	\end{lemma}
	
	\begin{proof}
		We first have $|\psi_R(x)| \lesssim 1$ since $\phi_R$ is bounded. On the other hand, thanks to the support of $\chi$, we see that if $|\tau|\geq 2R$, then $\phi_R(\tau)=0$. It follows that
		\begin{align} \label{est-psi}
		|\psi_R(x)| \lesssim \frac{1}{r} \int_0^{2R} |\phi_R(\tau)| d\tau \lesssim \frac{R}{|x|}. 
		\end{align}
		We thus prove the first estimate in \eqref{proper-psi-1}. The second equality in \eqref{proper-psi-1} follows from a direct computation. The first inequality in \eqref{proper-psi-2} comes from the fact $\chi$ is a decreasing function. The second estimate in \eqref{proper-psi-2} follows from the definition of $\phi_R$. For the third estimate in \eqref{proper-psi-2}, we have
		\begin{align*}
		|\phi_R(x) - \phi_{1,R}(x)| &\leq \frac{1}{\omega_3 R^3} \int \chi^2_R(x-z) \left|\chi^2_R(z) - \chi^{\alpha+2}_R(z)\right| dz \\
		&\leq \frac{1}{\omega_3} \int_{1-\eta \leq |z| \leq 1} \chi^2\left(\frac{x}{R}-z\right) |\chi^2(z) - \chi^{\alpha+2}(z)| dz \\
		&\lesssim \eta.
		\end{align*}
		The first estimate in \eqref{proper-psi-3} follows from the fact that
		\[
		\psi_R(x)|x| = \int_0^{2R}\phi_R(\tau) d\tau = R \int_0^2 \frac{1}{\omega_3} \int \chi^2(\tau-z) \chi^2 (z) dz d\tau.
		\]
		To see the second estimate in \eqref{proper-psi-3}, we consider two cases: $|x|\geq 2R$ and $|x|\leq 2R$. In the case $|x|\geq 2R$, it follows immediately from \eqref{est-psi} since $\phi_R(x)=0$. In the case $|x|\leq 2R$, we have
		\begin{align*}
		|\psi_R(x)-\phi_R(x)| &=\frac{1}{r} \left| \int_0^{r} \phi_R(\tau) - \phi_R(r) d\tau\right| \\
		&\lesssim \frac{1}{r} \left|\int_0^r \int_0^1 \phi'_R(r+\theta(\tau-r)) (\tau-r) d\theta d\tau \right| \\
		&\lesssim \frac{|x|}{R}.
		\end{align*}
		The last estimate in \eqref{proper-psi-3} is proved similarly by using the fact that
		\[
		\nabla \psi_R(x) = \frac{x}{|x|^2}(\phi_R(x) - \psi_R(x)).
		\]
		The proof is now complete.	
	\end{proof}
	
	\begin{lemma}[Coercivity I] \label{lem-coer-1}
		Let $\frac{4}{3}<\alpha<4$. Let $V:\R^3 \rightarrow \R$ satisfy \eqref{cond-1-V} and $V\geq 0$. Let $f \in H^1$ satisfy 
		\begin{align} \label{cond-f}
		\|f\|^{\alpha+2}_{L^{\alpha+2}} \|f\|^{2\sigc}_{L^2} \leq (1-\rho) \|Q\|^{\alpha+2}_{L^{\alpha+2}} \|Q\|^{2\sigc}_{L^2}
		\end{align}
		for some constant $\rho>0$. Then there exists $\nu=\nu(\rho)>0$ such that
		\begin{align} \label{est-f}
		\|\nabla f\|^2_{L^2} - \frac{3\alpha}{2(\alpha+2)} \|f\|^{\alpha+2}_{L^{\alpha+2}} \geq \nu \|\nabla f\|^2_{L^2}.
		\end{align}
	\end{lemma}
	
	\begin{proof}
		Multiplying both sides of \eqref{GN-ineq} with $\left[\|f\|^{\alpha+2}_{L^{\alpha+2}}\right]^{\frac{3\alpha-4}{4}}$ and using \eqref{poho-iden} together with \eqref{sharp-const-GN}, we have
		\begin{align*}
		\left[\|f\|^{\alpha+2}_{L^{\alpha+2}}\right]^{\frac{3\alpha}{4}} &\leq \left(\frac{2(\alpha+2)}{3\alpha}\right)^{\frac{3\alpha}{4}} \left(\frac{\|f\|^{\alpha+2}_{L^{\alpha+2}} \|f\|^{2\sigc}_{L^2}}{\|Q\|^{\alpha+2}_{L^{\alpha+2}} \|Q\|^{2\sigc}_{L^2}} \right)^{\frac{3\alpha-4}{4}} \|\nabla f\|^{\frac{3\alpha}{2}}_{L^2} \\
		&\leq \left(\frac{2(\alpha+2)}{3\alpha}\right)^{\frac{3\alpha}{4}} (1-\rho)^{\frac{3\alpha-4}{4}} \|\nabla f\|^{\frac{3\alpha}{2}}_{L^2}
		\end{align*}
		which implies
		\[
		\|f\|^{\alpha+2}_{L^{\alpha+2}}\leq \frac{2(\alpha+2)}{3\alpha} (1-\rho)^{\frac{3\alpha-4}{3\alpha}} \|\nabla f\|^2_{L^2}.
		\]
		We obtain
		\[
		\|\nabla f\|^2_{L^2} - \frac{3\alpha}{2(\alpha+2)} \|f\|^{\alpha+2}_{L^{\alpha+2}} \geq \left(1-(1-\rho)^{\frac{3\alpha-4}{3\alpha}} \right) \|\nabla f\|^2_{L^2}
		\]
		which shows \eqref{est-f} with $\nu= 1-(1-\rho)^{\frac{3\alpha-4}{3\alpha}}>0$. The proof is complete.
	\end{proof}
	
	\begin{lemma}[Coercivity II] \label{lem-coer-2}
		Let $\frac{4}{3}<\alpha<4$. Let $V:\R^3 \rightarrow \R$ satisfy \eqref{cond-1-V} and $V\geq 0$. Let $u$ be a $H^1$-solution to the focusing problem \eqref{NLS-V} satisfying \eqref{scat-cond}. Then $T^*=\infty$. Moreover, there exists $\nu=\nu(u_0,Q)>0$ such that for any $R>0$ and any $z,\xi \in \R^3$,
		\begin{align} \label{consequence-3}
		\|\nabla [\chi_R(\cdot-z)u^{\xi}(t)]\|^2_{L^2} - \frac{3\alpha}{2(\alpha+2)} \|\chi_R(\cdot-z) u^{\xi}(t)\|^{\alpha+2}_{L^{\alpha+2}} \geq \nu \|\nabla[\chi_R(\cdot-z) u^{\xi}(t)]\|^2_{L^2}
		\end{align}
		for all $t\in [0,\infty)$, where 
		\begin{align} \label{defi-u-xi}
		u^{\xi}(t,x):= e^{ix\cdot \xi} u(t,x).
		\end{align}
	\end{lemma}
	
	\begin{proof}
		First, it follows from \eqref{scat-cond}, the conservation of mass and energy that
		\[
		\sup_{t\in [0,T^*)} \|\nabla u(t)\|_{L^2} \leq C(u_0,Q)<\infty
		\]
		which, by the local theory, implies $T^*=\infty$. 
		
		Next, from \eqref{scat-cond}, we take $\rho=\rho(u_0,Q)>0$ such that
		\[
		\sup_{t\in [0,\infty)} \|u(t)\|^{\alpha+2}_{L^{\alpha+2}} \leq (1-\rho) \|Q\|^{\alpha+2}_{L^{\alpha+2}} \|Q\|^{2\sigc}_{L^2}.
		\]
		By the definitions of $\chi$ and $u^\xi$, we see that
		\[
		\|\chi_R(\cdot-z) u^\xi(t)\|_{L^2} \leq \|u(t)\|_{L^2}, \quad \|\chi_R(\cdot-z) u^\xi(t)\|_{L^{\alpha+2}} \leq \|u(t)\|_{L^{\alpha+2}}
		\]
		for all $t\in [0,\infty)$, all $R>0$ and all $z, \xi \in \R^3$. Thus, we get
		\[
		\sup_{t\in [0,\infty)} \|\chi_R (\cdot-z) u^\xi(t)\|^{\alpha+2}_{L^{\alpha+2}} \|\chi_R(\cdot-z) u^\xi(t)\|^{2\sigc}_{L^2} \leq (1-\rho) \|Q\|^{\alpha+2}_{L^{\alpha+2}} \|Q\|^{2\sigc}_{L^2}.
		\]
		Thanks to this estimate, \eqref{consequence-3} follows immediately from Lemma \ref{lem-coer-1}. 
	\end{proof}
	
	Let $u$ be a $H^1$-solution to \eqref{NLS-V} defined on the maximal forward time interval of existence $[0,T^*)$. We define the Morawetz action 
	\begin{align} \label{defi-mora-act}
	\Mca_R(t):= \int \psi_R(x) x \cdot 2 \ima(\overline{u}(t,x) \nabla u(t,x)) dx.
	\end{align}
	
	\begin{lemma} [Morawetz identity]
		Let $u$ be a $H^1$-solution to \eqref{NLS-V} satisfying
		\begin{align} \label{global-bound-A}
		\sup_{t\in [0,T^*)} \|u(t)\|_{H^1} \leq A
		\end{align}
		for some constant $A>0$. Let $\Mca_R(t)$ be as in \eqref{defi-mora-act}. Then we have 
		\begin{align} \label{est-1-deri-mora-act}
		\sup_{t\in [0,T^*)} |\Mca_R(t)| \lesssim_A R.
		\end{align}
		Moreover, we have
		\begin{align}
		\frac{d}{dt}\Mca_R(t) &= \mp \frac{2\alpha}{\alpha+2} \int \psi_R(x) x \cdot \nabla (|u(t,x)|^{\alpha+2}) dx  \label{term-1-mora} \\
		&\mathrel{\phantom{=}} + \int \psi_R(x) x \cdot \nabla \Delta (|u(t,x)|^2) dx \label{term-2-mora} \\
		&\mathrel{\phantom{=}} - 4 \sum_{j,k} \int \psi_R(x) x_j \partial_k[\rea(\partial_j u(t,x) \partial_k \overline{u}(t,x))] dx \label{term-3-mora} \\
		&\mathrel{\phantom{=}} - 2\int \psi_R(x) x \cdot \nabla V(x) |u(t,x)|^2 dx \label{term-4-mora}
		\end{align}
		for all $t\in [0,T^*)$.
	\end{lemma}
	
	\begin{proof}
	The estimate \eqref{est-1-deri-mora-act} follows directly from \eqref{proper-psi-1}, H\"older's inequality and \eqref{global-bound-A}. The identities \eqref{term-1-mora}--\eqref{term-4-mora} follow from a direct computation using the fact that
	\begin{align} \label{deri-time}
	\partial_t [2\ima(\overline{u}\partial_j u)] = - \sum_k \partial_k [4 \rea(\partial_j u \partial_k \overline{u}) - \delta_{jk} \Delta (|u|^2)] \mp \frac{2\alpha}{\alpha+2} \partial_j (|u|^{\alpha+2}) - 2 \partial_j V |u|^2
	\end{align}
	for $j=1,\cdots,3$, where $\delta_{jk}$ is the Kronecker symbol.
	\end{proof}
	
	\begin{lemma} [Morawetz estimate in the focusing case] \label{lem-mora-est-focus}
		Let $\frac{4}{3}<\alpha<4$. Let $V:\R^3 \rightarrow \R$ be radially symmetric satisfying \eqref{cond-1-V}, $V\geq 0$, $x \cdot \nabla V \leq 0$, and $\partial_r V \in L^q$ for any $\frac{3}{2} \leq q \leq \infty$. Let $u$ be a $H^1$-solution to the focusing problem \eqref{NLS-V} satisfying \eqref{scat-cond}. Define $\Mca_R(t)$ as in \eqref{defi-mora-act}. Then we have
		\begin{align} \label{est-V-radi-focus}
		\begin{aligned}
		-R\int \partial_r V |u(t)|^2 dx &\leq \frac{d}{dt} \Mca_R(t) + O(R^{-2}) - \int \nabla [3\phi_R(x)+2(\psi_R-\phi_R)(x)] \cdot \nabla (|u(t,x)|^2) dx \\
		&\mathrel{\phantom{\leq}} +\frac{6\alpha}{\alpha+2} \int (\phi_R-\phi_{1,R})(x) |u(t,x)|^{\alpha+2} dx + \frac{4\alpha}{\alpha+2} \int (\psi_R-\phi_R)(x) |u(t,x)|^{\alpha+2} dx
		\end{aligned}
		\end{align}
		for any $t\in [0,\infty)$.
	\end{lemma}
	
	\begin{remark} \label{rem-Lq}
		The condition $\partial_r V \in L^q$ for any $\frac{3}{2}\leq q \leq \infty$ is needed to ensure $\partial_r V |u(t)|^2 \in L^1$. In fact, 
		\begin{align*}
		\left| \int \partial_r V |u(t)|^2 dx \right| \leq \|\partial_r V\|_{L^q} \|u(t)\|^2_{L^{\frac{2q}{q-1}}} \leq C\|\partial_r V\|_{L^q} \|u(t)\|^2_{H^1},
		\end{align*}
		where we have used the Sobolev embedding $H^1(\R^3) \hookrightarrow L^{\frac{2q}{q-1}}(\R^3)$ for any $\frac{3}{2}\leq q \leq \infty$. Note that this condition can be relaxed to $\partial_r V \in L^q + L^\infty$ for some $q \geq \frac{3}{2}$.
	\end{remark}

	\noindent {\it Proof of Lemma \ref{lem-mora-est-focus}.}
		We first note that by Lemma \ref{lem-coer-2}, $T^*=\infty$ and \eqref{consequence-3} holds for all $t\in [0,\infty)$. 
		Note also that the condition $x\cdot \nabla V\leq 0$ is equivalent to $\partial_r V \leq 0$ since $V$ is radially symmetric. Using the fact that
		\begin{align} \label{deri-psi-j}
		\sum_j \partial_j (x_j \psi_R) = 3\psi_R + \sum_j x_j \partial_j \psi_R = 3\phi_R + 2 (\psi_R-\phi_R),
		\end{align}
		the integration by parts implies
		\begin{align}
		\eqref{term-1-mora} &=-\frac{2\alpha}{\alpha+2} \sum_j \int \partial_j [x_j \psi_R(x)] |u(t,x)|^{\alpha+2} dx \nonumber \\
		&= -\frac{2\alpha}{\alpha+2} \int [3\phi_R(x) + 2(\psi_R-\phi_R)(x)] |u(t,x)|^{\alpha+2} dx \nonumber \\
		&= -\frac{6\alpha}{\alpha+2} \int \phi_{1,R}(x) |u(t,x)|^{\alpha+2} dx \label{term-1-mora-1} \\
		&\mathrel{\phantom{=}} - \frac{6\alpha}{\alpha+2} \int (\phi_R- \phi_{1,R})(x) |u(t,x)|^{\alpha+2} dx \label{term-1-mora-2}\\
		&\mathrel{\phantom{=}} - \frac{4\alpha}{\alpha+2} \int (\psi_R-\phi_R)(x)|u(t,x)|^{\alpha+2} dx. \label{term-1-mora-3}
		\end{align}
		By the definition of $\phi_1$, we can write
		\begin{align}
		\eqref{term-1-mora-1} = -\frac{6\alpha}{(\alpha+2)\omega_3 R^3}  \iint \chi^2_R(z) \chi^{\alpha+2}_R(x-z) |u(t,x)|^{\alpha+2} dxdz. \label{term-1-mora-1-1}
		\end{align}
		We will consider \eqref{term-1-mora-2} and \eqref{term-1-mora-3} as error terms. By integrating by parts twice, we obtain that
		\begin{align}
		\eqref{term-2-mora} &= -\sum_{j,k} \int \partial_j [x_j \psi_R(x)] \partial^2_k (|u(t,x)|^2) dx \nonumber \\
		&= -\sum_k \int [3\phi_R(x) + 2(\psi_R-\phi_R)(x)] \partial^2_k (|u(t,x)|^2) dx \nonumber \\
		&= \sum_k \int \partial_k [3\phi_R(x) + 2(\psi_R-\phi_R)(x)] \partial_k (|u(t,x)|^2) dx. \label{term-2-mora-1}
		\end{align}
		To estimate \eqref{term-3-mora}, we denote
		\[
		P_{jk}(x):= \delta_{jk} - \frac{x_jx_k}{|x|^2}.
		\]
		By integrating by parts,
		\begin{align}
		\eqref{term-3-mora} &= 4\sum_{j,k} \int \partial_k[x_j \psi_R(x)] \rea ( \partial_j u(t,x) \partial_k \overline{u}(t,x)) dx \nonumber \\
		&=4 \sum_{j,k} \int \delta_{jk} \phi_R(x) \rea(\partial_j u(t,x) \partial_k \overline{u}(t,x)) dx \nonumber \\
		&\mathrel{\phantom{=}} + 4\sum_{j,k} \int P_{jk}(x) (\psi_R-\phi_R)(x) \rea(\partial_j u(t,x) \partial_k \overline{u}(t,x)) dx  \nonumber \\
		&= 4 \int \phi_R(x)|\nabla u(t,x)|^2 dx + 4\int (\psi_R-\phi_R)(x) |\nnabla u(t,x)|^2 dx \nonumber \\
		&\geq 4 \int \phi_R(x) |\nabla u(t,x)|^2 dx, \label{term-3-mora-1}
		\end{align}
		where we have used the fact $\psi_R-\phi_R \geq 0$ and
		\[
		\nnabla u(t,x):= \nabla u(t,x) - \frac{x}{|x|} \left( \frac{x}{|x|} u(t,x)\right)
		\]
		is the angular derivative. By the choice of $\phi$, we rewrite
		\begin{align}
		\eqref{term-3-mora-1} = \frac{4}{\omega_3 R^3} \iint \chi^2_R(z) \chi^2_R(x-z) |\nabla u(t,x)|^2 dx dz. \label{term-3-mora-1-1}
		\end{align}
		Since $V$ is radially symmetric, we have from \eqref{proper-psi-3} that
		\begin{align} 
		\eqref{term-4-mora} = - 2\int \psi_R(x) |x| \partial_r V |u(t,x)|^2 dx \sim -2 R \int \partial_rV |u(t,x)|^2 dx. \label{term-4-mora-1}
		\end{align}
		Collecting \eqref{term-1-mora-1}--\eqref{term-4-mora-1}, we get
		\begin{align*}
		\frac{d}{dt}\Mca_R(t) \geq &-\frac{6\alpha}{(\alpha+2) \omega_3 R^3}  \iint \chi^2_R(z) \chi^{\alpha+2}_R(x-z) |u(t,x)|^{\alpha+2} dx dz \\
		& -\frac{6\alpha}{\alpha+2} \int (\phi_R-\phi_{1,R})(x) |u(t,x)|^{\alpha+2} dx -\frac{4\alpha}{\alpha+2} \int (\psi_R-\phi_R)(x) |u(t,x)|^{\alpha+2} dx \\
		& + \int \nabla[3\phi_R(x) +2(\psi_R-\phi_R)(x)] \cdot \nabla (|u(t,x)|^2)dx \\
		&+ \frac{4}{\omega_3R^3} \iint \chi^2_R(z) \chi^2_R(x-z) |\nabla u(t,x)|^2 dx - 2R \int \partial_rV |u(t,x)|^2 dx.
		\end{align*}
		It follows that
		\begin{align*}
		\frac{4}{\omega_3 R^3} \iint &\chi^2_R(z) \left[ |\chi_R(x-z) \nabla u(t,x)|^2 - \frac{3\alpha}{2(\alpha+2)} |\chi_R(x-z) u(t,x)|^{\alpha+2} \right] dxdz - 2R\int\partial_rV |u(t,x)|^2 dx \\
		&\leq \frac{d}{dt} \Mca_R(t) +\frac{6\alpha}{\alpha+2} \int (\phi_R-\phi_{1,R})(x) |u(t,x)|^{\alpha+2} dx + \frac{4\alpha}{\alpha+2} \int (\psi_R-\phi_R)(x) |u(t,x)|^{\alpha+2} dx \\
		&\mathrel{\phantom{\leq \frac{d}{dt} \Mca_R(t)}} - \int \nabla [3\phi_R(x)+2(\psi_R-\phi_R)(x)] \cdot \nabla (|u(t,x)|^2) dx.
		\end{align*}
		For fixed $z\in \R^3$, we have from the fact
		\begin{align}
		\int |\nabla (\chi f)|^2 dx = \int \chi^2 |\nabla f|^2 dx - \int \chi \Delta \chi |f|^2 dx \label{prop-chi}
		\end{align}
		that
		\[
		\int |\chi_R(x-z) \nabla u(t,x)|^2 dx = \|\nabla [\chi_R(\cdot-z) u(t)]\|^2_{L^2} + O(R^{-2}\|u(t)\|^2_{L^2})
		\]
		which, by \eqref{consequence-3}, implies
		\begin{align*}
		\int |\chi_R(x-z) \nabla u(t,x)|^2 dx &-\frac{3\alpha}{2(\alpha+2)} \int |\chi_R(x-z) u(t,x)|^{\alpha+2} dx \\
		&= \|\nabla [\chi_R(\cdot-z) u(t)]\|^2_{L^2} - \frac{3\alpha}{2(\alpha+2)} \|\chi_R(\cdot-z) u(t)\|^{\alpha+2}_{L^\alpha+2} + O(R^{-2}) \\
		&\geq \nu \|\nabla [\chi_R(\cdot-z) u(t)]\|^2_{L^2} + O(R^{-2})
		\end{align*}
		for all $t\in [0,\infty)$. Thus, we obtain
		\begin{align*}
		-R\int \partial_rV |u(t,x)|^2 dx &\leq \frac{d}{dt} \Mca_R(t)  + O(R^{-2}) - \int \nabla [3\phi_R(x)+2(\psi_R-\phi_R)(x)] \cdot \nabla (|u(t,x)|^2) dx \\
		&\mathrel{\phantom{\leq}} +\frac{6\alpha}{\alpha+2} \int (\phi_R-\phi_{1,R})(x) |u(t,x)|^{\alpha+2} dx + \frac{4\alpha}{\alpha+2} \int (\psi_R-\phi_R)(x) |u(t,x)|^{\alpha+2} dx
		\end{align*}
		which proves \eqref{est-V-radi-focus}.
	\hfill $\Box$
	
	By the same argument as in the proof Lemma \ref{lem-mora-est-focus} (but even simpler), we get the following result in the defocusing case.
	\begin{corollary} [Morawetz estimate in the defocusing case] \label{coro-mora-est-defocus}
		Let $\frac{4}{3}<\alpha<4$. Let $V:\R^3 \rightarrow \R$ be radially symmetric satisfying \eqref{cond-1-V}, \eqref{cond-2-V}, $x \cdot \nabla V \leq 0$, and $\partial_r V \in L^q$ for any $\frac{3}{2} \leq q \leq \infty$. Let $u_0 \in H^1$ and $u$ be the corresponding global solution to the defocusing problem \eqref{NLS-V}. Define $\Mca_R(t)$ as in \eqref{defi-mora-act}. Then we have
		\begin{align} 
		-R\int \partial_r V |u(t)|^2 dx &\leq \frac{d}{dt} \Mca_R(t) - \int \nabla [3\phi_R(x)+2(\psi_R-\phi_R)(x)] \cdot \nabla (|u(t,x)|^2) dx \label{est-V-radi-defocus}\\
		&\mathrel{\phantom{\leq}} -\frac{6\alpha}{\alpha+2} \int (\phi_R-\phi_{1,R})(x) |u(t,x)|^{\alpha+2} dx - \frac{4\alpha}{\alpha+2} \int (\psi_R-\phi_R)(x) |u(t,x)|^{\alpha+2} dx \nonumber
		\end{align}
		for any $t\in \R$.
	\end{corollary}

	We next define the interaction Morawetz action
	\begin{align} \label{defi-M-R}
	\Mca^{\otimes 2}_R(t):= \iint |u(t,y)|^2 \psi_R(x-y) (x-y) \cdot 2\ima \left( \overline{u}(t,x) \nabla u(t,x)\right) dx dy.
	\end{align}
	
	\begin{lemma} [Interaction Morawetz identity] 
		Let $u$ be a $H^1$-solution to \eqref{NLS} satisfying
		\begin{align*} 
		\sup_{t\in [0,T^*)} \|u(t)\|_{H^1} \leq A
		\end{align*}
		for some constant $A>0$. Let $\Mca_R(t)$ be as in \eqref{defi-M-R}. Then we have
		\begin{align} \label{est-M-R}
		\sup_{t\in [0,T^*)} |\Mca^{\otimes 2}_R(t)| \lesssim_A R.
		\end{align}
		Moreover, we have
		\begin{align}
		\frac{d}{dt} \Mca^{\otimes 2}_R(t) &= \mp \frac{2\alpha}{\alpha+2} \iint |u(t,y)|^2 \psi_R(x-y) (x-y) \cdot \nabla (|u(t,x)|^{\alpha+2}) dx dy \label{term-1} \\
		&\mathrel{\phantom{=}}+ \iint |u(t,y)|^2 \psi_R(x-y) (x-y) \cdot \nabla \Delta(|u(t,x)|^2) dx dy \label{term-2} \\
		&\mathrel{\phantom{=}} - 4 \sum_{j,k}\iint \partial_j \left[\ima (\overline{u}(t,y) \nabla u(t,y)) \right] \psi_R(x-y) (x_k-y_k) \ima (\overline{u}(t,x) \partial_k u(t,x)) dx dy \label{term-3}\\
		&\mathrel{\phantom{=}} - 4 \sum_{j,k} \iint |u(t,y)|^2 \psi_R(x-y) (x_j-y_j) \partial_k \left[ \rea (\partial_j u(t,x) \partial_k \overline{u}(t,x)) \right] dx dy \label{term-4} \\
		&\mathrel{\phantom{=}} -2 \iint |u(t,y)|^2 \psi_R(x-y) (x-y)\cdot \nabla V(x) |u(t,x)|^2 dx dy \label{term-5}
		\end{align}
		for all $t\in [0,T^*)$.
	\end{lemma}
	
	\begin{proof}
		The estimate \eqref{est-M-R} follows directly from \eqref{proper-psi-1} and H\"older's inequality. The identities \eqref{term-1}--\eqref{term-5} follow from a direct computation using 
		\begin{align*}
		\partial_t (|u|^2) = -\sum_{j} \partial_j[2\ima(\overline{u} \partial_j u)]
		\end{align*}
		and \eqref{deri-time}.
	\end{proof}
	
	\begin{proposition} [Interaction Morawetz estimate in the focusing case] \label{prop-inter-mora-est-focus}
		Let $\frac{4}{3}<\alpha<4$. Let $V:\R^3 \rightarrow \R$ be radially symmetric satisfying \eqref{cond-1-V}, $V\geq 0$, $x\cdot \nabla V \leq 0$, and $\partial_r V \in L^q$ for any $\frac{3}{2} \leq q \leq \infty$. Let $u$ be a $H^1$-solution to the focusing problem \eqref{NLS-V} satisfying \eqref{scat-cond}. Define $\Mca^{\otimes 2}_R(t)$ as in \eqref{defi-M-R}. Then for $\vareps>0$ sufficiently small, there exist $T_0 = T_0(\vareps)$, $J=J(\vareps)$, $R_0= R_0(\vareps, u_0,Q)$ sufficiently large and $\eta=\eta(\vareps)>0$ sufficiently small such that for any $a \in \R$,
		\begin{align} \label{inter-mora-est}
		\frac{1}{JT_0} \int_a^{a+T_0}\int_{R_0}^{R_0e^J} \frac{1}{R^3} \iiint \left|\chi_R(y-z)  u(t,y)\right|^2 \left| \nabla \left[ \chi_R(x-z) u^{\xi}(t,x) \right] \right|^2 dx dy dz \frac{dR}{R} dt \lesssim \vareps,
		\end{align}
		where $\chi_R(x)=\chi(x/R)$ with $\chi$ as in \eqref{defi-chi} and $u^\xi$ is as in \eqref{defi-u-xi} with some $\xi= \xi(t,z,R) \in \R^3$.
	\end{proposition}
	
	\begin{proof}
		By integrating by parts and using \eqref{deri-psi-j}, we have
		\begin{align}
		\eqref{term-1} &= - \frac{2\alpha}{\alpha+2} \sum_{j} \iint |u(t,y)|^2 \partial_j [ (x_j-y_j) \psi_R(x-y)] |u(t,x)|^{\alpha+2} dxdy \nonumber \\
		&=-\frac{6\alpha}{\alpha+2} \iint |u(t,y)|^2 \phi_R(x-y) |u(t,x)|^{\alpha+2} dx dy \nonumber \\
		&\mathrel{\phantom{=}} - \frac{4\alpha}{\alpha+2} \iint |u(t,y)|^2 (\psi_R-\phi_R)(x-y) |u(t,x)|^{\alpha+2} dx dy \nonumber \\
		&= -\frac{6\alpha}{\alpha+2} \iint |u(t,y)|^2 (\phi_R-\phi_{1,R})(x-y) |u(t,x)|^{\alpha+2} dx dy \label{term-1-1} \\
		&\mathrel{\phantom{=}} - \frac{4\alpha}{\alpha+2} \iint |u(t,y)|^2 (\psi_R-\phi_R)(x-y) |u(t,x)|^{\alpha+2} dx dy \label{term-1-2} \\
		&\mathrel{\phantom{=}} - \frac{6\alpha}{\alpha+2} \iint |u(t,y)|^2 \phi_{1,R}(x-y) |u(t,x)|^{\alpha+2} dx dy, \label{term-1-3}
		\end{align}
		where $\phi_{1,R}$ is as in \eqref{defi-phi-1}. We will consider \eqref{term-1-1} and \eqref{term-1-2} as error terms. Moreover, we use the fact
		\[
		\phi_{1,R}(x-y) = \frac{1}{\omega_3 R^3} \int \chi^2_R(x-y-z) \chi^{\alpha+2}_R(z)dz = \frac{1}{\omega_3 R^3} \int \chi^2_R(y-z) \chi^{\alpha+2}_R(x-z) dz
		\]
		to write
		\begin{align} \label{term-1-3-1}
		\eqref{term-1-3} = -\frac{6\alpha}{(\alpha+2)\omega_3 R^3} \iiint \chi^2_R(y-z) \chi^{\alpha+2}_R(x-z) |u(t,y)|^2 |u(t,x)|^{\alpha+2} dx dydz. 
		\end{align}
		By integrating by parts twice and \eqref{deri-psi-j}, we see that
		\begin{align}
		\eqref{term-2} &= \sum_{j,k} \iint |u(t,y)|^2 \psi_R(x-y) (x_j-y_j) \partial_j \partial^2_k (|u(t,x)|^2) dx dy \nonumber \\
		&= - \sum_{j,k} \iint |u(t,y)|^2 \partial_j[\psi_R(x-y) (x_j-y_j)] \partial^2_k (|u(t,x)|^2) dx dy  \nonumber \\
		&= \sum_k \iint |u(t,y)|^2 \partial^x_k \left[3\phi_R(x-y) +2(\psi_R-\phi_R)(x-y) \right] \partial_k(|u(t,x)|^2) dx dy, \label{term-2-1} 
		\end{align}
		where $\partial^x_k$ is $\partial_k$ with respect to the $x$-variable. We next consider \eqref{term-3} and \eqref{term-4}. To this end, we denote
		\[
		P_{jk}(x-y) := \delta_{jk} - \frac{(x_j-y_j)(x_k-y_k)}{|x-y|^2}.
		\]
		By integration by parts, we have
		\begin{align}
		\eqref{term-3} &= - 4\sum_{j,k} \iint \partial_j[ \ima(\overline{u}(t,y) \partial_j u(t,y)] \psi_R(x-y) (x_k-y_k) \ima ( \overline{u}(t,x) \partial_k u(t,x)) dx dy \nonumber \\
		&=4 \sum_{j,k} \iint \ima(\overline{u}(t,y) \partial_j u(t,y)) \partial^y_j [\psi_R(x-y) (x_k-y_k)] \ima(\overline{u}(t,x) \partial_k(t,x)) dx dy \nonumber \\
		&=-4\sum_{j,k} \iint \ima(\overline{u}(t,y) \partial_j u(t,y)) \delta_{jk} \phi_R(x-y) \ima(\overline{u}(t,x) \partial_k u(t,x)) dx dy \label{term-3-1} \\
		&\mathrel{\phantom{=}} - 4\sum_{jk} \iint \ima(\overline{u}(t,y) \partial_j u(t,y)) P_{jk}(x-y) (\psi_R-\phi_R)(x-y) \ima(\overline{u}(t,x) \partial_k u(t,x)) dxdy, \label{term-3-2}
		\end{align}
		where $\partial^y_j$ is $\partial_j$ with respect to the $y$-variable. Similarly,
		\begin{align}
		\eqref{term-4} &=- 4\sum_{j,k} \iint |u(t,y)|^2 \psi_R(x-y) (x_j-y_j) \partial_k \left[\rea(\partial_j u(t,x) \partial_k \overline{u}(t,x)) \right] dx dy \nonumber \\
		&=4 \sum_{j,k} \iint |u(t,y)|^2 \partial^x_k [\psi_R(x-y) (x_j-y_j)] \rea(\partial_j u(t,x) \partial_k \overline{u}(t,x)) dx dy \nonumber \\
		&=4 \sum_{j,k} \iint |u(t,y)|^2 \delta_{jk} \phi_R(x-y) \rea(\partial_j u(t,x) \partial_k \overline{u}(t,x)) dx dy \label{term-4-1} \\
		&\mathrel{\phantom{=}}+ 4\sum_{j,k} \iint |u(t,y)|^2 P_{jk}(x-y) (\psi_R-\phi_R)(x-y) \rea(\partial_j u(t,x) \partial_k \overline{u}(t,x)) dx dy. \label{term-4-2}
		\end{align}
		We see that
		\begin{align*}
		\eqref{term-3-2} + \eqref{term-4-2} &= 4\iint |u(t,y)|^2 |\nnabla_y u(t,x)|^2 (\psi_R-\phi_R)(x-y) dx dy \\
		&\mathrel{\phantom{=}} - 4\iint \ima( \overline{u}(t,y) \nnabla_x u(t,y)) \cdot \ima( \overline{u}(t,x) \nnabla_y u(t,x)) (\psi_R-\phi_R)(x-y) dx dy,
		\end{align*}
		where
		\[
		\nnabla_y u(t,x):= \nabla u(t,x) - \frac{x-y}{|x-y|} \left(\frac{x-y}{|x-y|} \nabla u(t,x) \right)
		\]
		is the angular derivative centered at $y$, and similarly for $\nnabla_x u(t,y)$. By Cauchy-Schwarz inequality and the fact that $\psi-\phi$ is non-negative, we deduce 
		\begin{align} \label{term-3+4-2}
		\eqref{term-3-2}+\eqref{term-4-2} \geq 0.
		\end{align}
		We next have 
		\begin{align*}
		\eqref{term-3-1} +\eqref{term-4-1} = 4 \iint \phi_R(x-y) \left[|u(t,y)|^2 |\nabla u(t,x)|^2 - \ima(\overline{u}(t,y) \nabla u(t,y)) \cdot \ima(\overline{u}(t,x) \nabla u(t,x)) \right] dx dy.
		\end{align*}
		Using the fact that
		\[
		\phi_R(x-y) = \frac{1}{\omega_3 R^3} \int \chi^2_R(x-y-z) \chi^2_R(z) dz = \frac{1}{\omega_3 R^3} \int \chi^2_R(x-z) \chi^2_R(y-z)dz,
		\]
		we get
		\begin{multline*}
		\eqref{term-3-1} +\eqref{term-4-1} =\frac{4}{\omega_3 R^3} \iiint \chi^2_R(x-z) \chi^2_R(y-z) \\
		\times \Big[ |u(t,y)|^2 |\nabla u(t,x)|^2 - \ima (\overline{u}(t,y) \nabla u(t,y)) \cdot \ima (\overline{u}(t,x) \nabla u(t,x)) \Big] dx dy dz.
		\end{multline*}
		For fixed $z \in \R^3$, we consider the quantity defined by
		\[
		\iint \chi^2_R(x-z)  \chi^2_R (y-z)\left[ |u(t,y)|^2 |\nabla u(t,x)|^2 - \ima (\overline{u}(t,y) \nabla u(t,y)) \cdot \ima (\overline{u}(t,x) \nabla u(t,x)) \right] dxdy.
		\]
		We claim that this quantity is invariant under the Galilean transformation
		\[
		u(t,x) \mapsto u^\xi(t,x) := e^{ix \cdot \xi} u(t,x)
		\]
		for any $\xi = \xi(t,z,R)$. Indeed, one has
		\begin{align*}
		|u^\xi(y)|^2 |\nabla u^\xi(x)|^2 &- \ima( \overline{u}^\xi(y) \nabla u^\xi(y)) \cdot \ima( \overline{u}^\xi(x) \nabla u^\xi(x)) \\
		&= |u(y)|^2 |\nabla u(x)|^2 - \ima( \overline{u}(y) \nabla u(y)) \cdot \ima(\overline{u}(x) \nabla u(x)) \\
		&\mathrel{\phantom{=}} + \xi \cdot |u(y)|^2 \ima(\overline{u}(x)\nabla u(x)) - \xi \cdot |u(x)|^2 \ima(\overline{u}(y) \nabla u(y))
		\end{align*}
		and hence the claim follows by symmetry of $\chi$ and a change of variable. We now define $\xi = \xi (t,z,R)$ so that
		\[
		\int \chi^2_R(x-z) \ima(\overline{u}^\xi(t,x) \nabla u^\xi(t,x)) dx =0.
		\]
		In particular, we can achieve this by choosing
		\[
		\xi(t,z,R)= - \int \chi^2_R(x-z) \ima(\overline{u}(t,x) \nabla u(t,x)) dx \div \int \chi^2_R(x-z) |u(t,x)|^2 dx
		\]
		provided the denominator is non-zero (otherwise $\xi\equiv 0$ suffices). 
		
		For this choice of $\xi$, we have
		\begin{align} \label{term-3+4-1}
		\eqref{term-3-1} +\eqref{term-4-1} = \frac{4}{\omega_3 R^3} \iiint \chi^2_R(x-z) \chi^2_R(y-z)|u(t,y)|^2 |\nabla u^\xi(t,x)|^2 dx dy dz. 
		\end{align}
		We next estimate \eqref{term-5}. Since $V$ is radially symmetric and $x\cdot \nabla V \leq 0$, we write
		\begin{align}
			\eqref{term-5} = -2 \iint |u(t,y)|^2 \psi_R(x-y) (x-y)\cdot \frac{x}{|x|} \partial_rV |u(t,x)|^2 dx dy. \label{term-5-cas-1}
		\end{align}
		
		Collecting \eqref{term-1-1}, \eqref{term-1-2}, \eqref{term-1-3-1}, \eqref{term-2-1}, \eqref{term-3+4-2}, \eqref{term-3+4-1} and \eqref{term-5-cas-1}, we obtain
		\begin{align*}
		\frac{d}{dt} \Mca^{\otimes 2}_R(t) \geq &-\frac{6\alpha}{\alpha+2} \iint |u(t,y)|^2 (\phi_R-\phi_{1,R})(x-y) |u(t,x)|^{\alpha+2} dx dy \\
		&-\frac{4\alpha}{\alpha+2} \iint |u(t,y)|^2 (\psi_R-\phi_R)(x-y) |u(t,x)|^{\alpha+2} dx dy \\
		&-\frac{6\alpha}{(\alpha+2)\omega_3 R^3} \iiint |\chi_R(y-z) u(t,y)|^2 |\chi_R(x-z) u(t,x)|^{\alpha+2} dxdydz \\
		&+ \iint |u(t,y)|^2 \nabla \left[ 3\phi_R(x-y) + 2(\psi_R-\phi_R)(x-y)\right] \cdot \nabla (|u(t,x)|^2) dx dy \\
		&+\frac{4}{\omega_3 R^3} \iiint |\chi_R(y-z) u(t,y)|^2 |\chi_R(x-z) \nabla u^\xi(t,x)|^2 dx dy dz \\
		&-2 \iint |u(t,y)|^2 \psi_R(x-y) (x-y)\cdot \frac{x}{|x|} \partial_rV |u(t,x)|^2 dx dy.
		\end{align*}
		It follows that
		\begin{align}
		\frac{4}{\omega_3R^3} \iiint &|\chi_R(y-z) u(t,y)|^2  \left[|\chi_R(x-z)\nabla u^\xi(t,x)|^2 - \frac{3\alpha}{2(\alpha+2)}|\chi_R(x-z) u(t,x)|^{\alpha+2} \right] dxdydz \label{lhs} \\
		& \leq \frac{d}{dt} \Mca^{\otimes 2}_R(t) + \frac{4\alpha}{\alpha+2} \iint |u(t,y)|^2 (\psi_R-\phi_R)(x-y) |u(t,x)|^{\alpha+2} dx dy \nonumber \\
		&\mathrel{\phantom{\leq \frac{d}{dt} \Mca^{\otimes 2}_R(t)}} + \frac{6\alpha}{\alpha+2} \iint |u(t,y)|^2 (\phi_R-\phi_{1,R})(x-y) |u(t,x)|^{\alpha+2} dx dy \nonumber \\
		&\mathrel{\phantom{\leq \frac{d}{dt} \Mca^{\otimes 2}_R(t)}} - \iint |u(t,y)|^2 \nabla \left[3\phi_R(x-y) + 2 (\psi_R-\phi_R)(x-y) \right] \cdot \nabla (|u(t,x)|^2) dx dy \nonumber \\
		&\mathrel{\phantom{\leq \frac{d}{dt} \Mca^{\otimes 2}_R(t)}} + 2\iint |u(t,y)|^2 \psi_R(x-y) (x-y)\cdot \frac{x}{|x|} \partial_rV |u(t,x)|^2 dx dy. \label{term-6}
		\end{align}
		Let us estimate the terms appeared from the second to the fifth lines. 	By \eqref{est-M-R}, we see that
		\begin{align}
		\left| \frac{1}{JT_0} \int_a^{a+T_0} \int_{R_0}^{R_0 e^J} \frac{d}{dt} \Mca^{\otimes 2}_R(t) \frac{dR}{R} dt \right| &\leq \frac{1}{JT_0} \int_{R_0}^{R_0e^J} \sup_{t\in [a,a+T_0]} |\Mca^{\otimes 2}_R(t)| \frac{dR}{R} \nonumber \\
		&\lesssim \frac{1}{JT_0} \int_{R_0}^{R_0 e^J} dR \lesssim \frac{R_0e^J}{JT_0}. \label{est-term-1}
		\end{align}
		Using \eqref{proper-psi-3}, we see that
		\begin{align}
		\Big|\frac{1}{JT_0} \int_a^{a+T_0} \int_{R_0}^{R_0e^J} &\iint |u(t,y)|^2 (\psi_R-\phi_R)(x-y) |u(t,x)|^{\alpha+2} dxdy \frac{dR}{R} dt \Big| \nonumber \\
		&\lesssim \frac{1}{JT_0} \int_a^{a+T_0} \int_{R_0}^{R_0e^J} \iint |u(t,y)|^2 \min \left\{ \frac{|x-y|}{R}, \frac{R}{|x-y|}\right\} |u(t,x)|^{\alpha+2} dx dy \frac{dR}{R} dt \nonumber \\
		&\lesssim \frac{1}{JT_0} \int_a^{a+T_0} \iint |u(t,y)|^2 |u(t,x)|^{\alpha+2} \left( \int_{R_0}^{R_0e^J} \min \left\{ \frac{|x-y|}{R}, \frac{R}{|x-y|}\right\} \frac{dR}{R} \right) dx dy dt \nonumber \\
		&\lesssim \frac{1}{J}. \label{est-term-2}
		\end{align}
		Here we have used the fact that $\sup_{t\in \R} \|u(t)\|_{H^1} \leq C(u_0,Q)<\infty$ and 
		\begin{align} \label{int-R}
		\int_{R_0}^{R_0e^J} \min \left\{ \frac{|x-y|}{R}, \frac{R}{|x-y|}\right\} \frac{dR}{R} \lesssim 1.
		\end{align}
		To see \eqref{int-R}, we have
		\begin{align*}
		\text{LHS}\eqref{int-R} &= \int_{R_0}^{R_0e^J} \frac{|x-y|}{R} \mathds{1}_{\{|x-y|\leq R\}} \frac{dR}{R} + \int_{R_0}^{R_0e^J} \frac{R}{|x-y|} \mathds{1}_{\{|x-y|\geq R\}} \frac{dR}{R} \\
		&= \int_{\max\{|x-y|,R_0\}}^{R_0e^J} \frac{|x-y|}{R^2}dR + \int_{R_0}^{\min \{|x-y|,R_0e^J\}} \frac{dR}{|x-y|} \\
		&= |x-y| \left(\frac{1}{\max\{|x-y|,R_0\}} - \frac{1}{R_0e^J}\right) + \frac{1}{|x-y|} \left( \min \{|x-y|, R_0e^J\} -R_0\right) \\
		&\lesssim 1.
		\end{align*}
		Using \eqref{proper-psi-2}, we have
		\begin{align}
		\Big|\frac{1}{JT_0} \int_a^{a+T_0} \int_{R_0}^{R_0e^J} &\iint |u(t,y)|^2 (\phi_R-\phi_{1,R})(x-y) |u(t,x)|^{\alpha+2} dx dy \frac{dR}{R} dt \Big| \nonumber \\
		&\lesssim \frac{1}{JT_0} \int_a^{a+T_0} \int_{R_0}^{R_0e^J} \eta \frac{dR}{R} dt \nonumber \\
		&\lesssim \eta. \label{est-term-3}
		\end{align}
		Using the fact $|\nabla \phi_R(x)| \lesssim \frac{1}{R}$, we see that
		\begin{align}
		\Big| \frac{1}{JT_0} \int_a^{a+T_0} \int_{R_0}^{R_0e^J} &\iint |u(t,y)|^2 \nabla \phi_R(x-y) \cdot \nabla (|u(t,x)|^2) dx dy \frac{dR}{R} dt \Big| \nonumber \\
		&\lesssim \frac{1}{JT_0} \int_a^{a+T_0} \int_{R_0}^{R_0e^J} \|u(t)\|^3_{L^2} \|\nabla u(t)\|_{L^2} \frac{dR}{R^2} dt \nonumber \\
		&\lesssim \frac{1}{JT_0} \int_a^{a+T_0} \int_{R_0}^{R_0e^J} \frac{dR}{R^2} dt \nonumber \\
		&\lesssim \frac{1}{JR_0}. \label{est-term-4}
		\end{align}
		Similarly, as $|\nabla(\psi_R-\phi_R)(x)| \lesssim \min \left\{\frac{1}{R},\frac{R}{|x|^2} \right\} <\frac{1}{R}$, we have
		\begin{align}
		\Big| \frac{1}{JT_0} \int_a^{a+T_0} \int_{R_0}^{R_0e^J}  \iint |u(t,y)|^2 \nabla (\psi_R-\phi_R)(x-y) \cdot \nabla(|u(t,x)|^2) dx dy \frac{dR}{R} dt\Big| \lesssim \frac{1}{JR_0}. \label{est-term-5}
		\end{align}
		We next consider the term in \eqref{term-6}. Note that this term does not appear in the case $V$ is non-radially symmetric. By H\"older's inequality and the conservation of mass, we have
		\[
		\Big| \frac{1}{JT_0} \int_a^{a+T_0} \int_{R_0}^{R_0e^J} \eqref{term-6} \frac{dR}{R} dt \Big| \lesssim  \Big| \frac{1}{JT_0} \int_a^{a+T_0} \int_{R_0}^{R_0e^J} \left(-R\int \partial_rV |u(t,x)|^2 dx\right) \frac{dR}{R} dt \Big|.
		\]
		We then use the Morawetz estimate given in Lemma $\ref{lem-mora-est-focus}$ to get
		\begin{align}
		\Big| \frac{1}{JT_0} \int_a^{a+T_0} \int_{R_0}^{R_0e^J} \eqref{term-6} \frac{dR}{R} dt \Big| &\lesssim  \Big|\frac{1}{JT_0} \int_a^{a+T_0} \int_{R_0}^{R_0 e^J} \frac{d}{dt} \Mca_R(t) \frac{dR}{R} dt \Big| \label{est-term-6-1} \\
		&\mathrel{\phantom{\lesssim}} + \Big|\frac{1}{JT_0} \int_a^{a+T_0} \int_{R_0}^{R_0 e^J} O(R^{-2}) \frac{dR}{R} dt \Big| \label{est-term-6-2} \\
		&\mathrel{\phantom{\lesssim}} + \Big|\frac{1}{JT_0} \int_a^{a+T_0} \int_{R_0}^{R_0 e^J} \int \nabla \phi_R(x)\cdot \nabla (|u(t,x)|^2) dx \frac{dR}{R} dt \Big| \label{est-term-6-3} \\
		&\mathrel{\phantom{\lesssim}} + \Big|\frac{1}{JT_0} \int_a^{a+T_0} \int_{R_0}^{R_0 e^J} \int \nabla (\psi_R-\phi_R)(x)\cdot \nabla (|u(t,x)|^2) dx \frac{dR}{R} dt \Big| \label{est-term-6-4} \\
		&\mathrel{\phantom{\lesssim}} + \Big|\frac{1}{JT_0} \int_a^{a+T_0} \int_{R_0}^{R_0 e^J} \int (\phi_R-\phi_{1,R})(x) |u(t,x)|^{\alpha+2} dx \frac{dR}{R} dt \Big| \label{est-term-6-5} \\
		&\mathrel{\phantom{\lesssim}} + \Big|\frac{1}{JT_0} \int_a^{a+T_0} \int_{R_0}^{R_0 e^J} \int (\psi_R-\phi_R)(x) |u(t,x)|^{\alpha+2} dx \frac{dR}{R} dt \Big|. \label{est-term-6-6} 
		\end{align}
		The term \eqref{est-term-6-1} is estimated as for \eqref{est-term-1} using \eqref{est-1-deri-mora-act}. The terms \eqref{est-term-6-2} and \eqref{est-term-6-3} are treated as for \eqref{est-term-4}. The terms \eqref{est-term-6-4}, \eqref{est-term-6-5} and \eqref{est-term-6-6} are respectively estimated as for \eqref{est-term-5}, \eqref{est-term-3} and \eqref{est-term-2}. 
		
		Combining \eqref{lhs}, \eqref{est-term-1}, \eqref{est-term-2}, \eqref{est-term-3}, \eqref{est-term-4} and \eqref{est-term-5}--\eqref{est-term-6-6}, we obtain 
		\begin{align}
		\Big|\frac{1}{JT_0} \int_a^{a+T_0} \int_{R_0}^{R_0e^J} \frac{1}{R^N} &\iiint |\chi_R(y-z) u(t,y)|^2 \nonumber \\
		&\times \Big[ |\chi_R(x-z)\nabla u^\xi(t,x)|^2 - \frac{N\alpha}{2(\alpha+2)} |\chi_R(x-z) u(t,x)|^{\alpha+2} \Big] dxdydz \frac{dR}{R} dt \Big| \nonumber \\
		&\lesssim \frac{R_0e^J}{JT_0} + \frac{1}{J} +\eta + \frac{1}{JR_0}. \label{est-lhs}
		\end{align}
		Now, for fixed $z, \xi \in \R^N$, we have from \eqref{prop-chi} that
		\[
		\int |\chi_R(x-z) \nabla u^\xi(t,x)|^2 dx =\|\nabla[ \chi_R(\cdot-z) u^{\xi}(t)]\|^2_{L^2} + O(R^{-2} \|u(t)\|^2_{L^2}).
		\]
		It follows that from the conservation of mass and \eqref{consequence-3} that for $R\geq R_0$ with $R_0$ sufficiently large,
		\begin{align*}
		\int |\chi_R(x-z) \nabla u^{\xi}(t,x)|^2 dx &- \frac{N\alpha}{2(\alpha+2)} \int |\chi_R(x-z) u(t,x)|^{\alpha+2} dx  \\
		&= \|\nabla[\chi_R(\cdot-z) u^{\xi}(t)]\|^2_{L^2} - \frac{N\alpha}{2(\alpha+2)} \|\chi_R(\cdot-z) u^{\xi}(t)\|^{\alpha+2}_{L^{\alpha+2}} + O(R^{-2}) \\
		&\geq \nu \|\nabla[\chi_R(\cdot-z) u^{\xi}(t)]\|^2_{L^2} + O(R^{-2}).
		\end{align*}
		The term $O(R^{-2})$ can be treated as in \eqref{est-term-4}. We thus infer from \eqref{est-lhs} that
		\begin{align*}
		\Big| \frac{1}{JT_0} \int_a^{a+T_0} \int_{R_0}^{R_0e^J} \frac{1}{R^N} \iiint |\chi_R(y-z) u(t,y)|^2 &|\nabla[\chi_R(x-z) u^\xi(t,x)]|^2 dxdydz \frac{dR}{R} dt \Big| \\
		&\lesssim \frac{R_0e^J}{JR_0} +\frac{1}{J} +\eta +\frac{1}{JR_0}.
		\end{align*}
		This proves \eqref{inter-mora-est} by taking $\eta=\vareps, J=\vareps^{-2}$, $R_0=\vareps^{-1}$ and $T_0=e^{\vareps^{-2}}$.
	\end{proof}
	
	Performing the same arguments as above, we get the following interaction Morawetz estimate in the defocusing case.
	\begin{corollary} [Interaction Morawetz estimate in the defocusing case] \label{coro-inter-mora-est-defocus}
		Let $\frac{4}{3}<\alpha<4$. Let $V:\R^3 \rightarrow \R$ satisfy \eqref{cond-1-V}, \eqref{cond-2-V}, and \eqref{cond-3-V}. Let $u_0 \in H^1$ and $u$ be the corresponding global solution to the defocusing problem \eqref{NLS-V}. Define $\Mca^{\otimes 2}_R(t)$ as in \eqref{defi-M-R}. Then for $\vareps>0$ sufficiently small, there exist $T_0 = T_0(\vareps)$, $J=J(\vareps)$, $R_0= R_0(\vareps)$ sufficiently large and $\eta=\eta(\vareps)>0$ sufficiently small such that for any $a \in \R$,
		\begin{align} \label{inter-mora-est-defocus}
		\frac{1}{JT_0} \int_a^{a+T_0}\int_{R_0}^{R_0e^J} \frac{1}{R^3} \iiint \left|\chi_R(y-z)  u(t,y)\right|^2 \left| \nabla \left[ \chi_R(x-z) u^\xi(t,x) \right] \right|^2 dx dy dz \frac{dR}{R} dt \lesssim \vareps,
		\end{align}
		where $\chi_R(x)=\chi(x/R)$ with $\chi$ as in \eqref{defi-chi}, and $u^\xi$ is as in \eqref{defi-u-xi} with some $\xi= \xi(t,z,R) \in \R^3$.
	\end{corollary}
	
	\begin{proof}
	The proof is similar to that of Proposition \ref{prop-inter-mora-est-focus} except for the term \eqref{term-5}. To treat this term, we consider two cases. 
	
	{\bf Case 1:} $V$ is radially symmetric and $x\cdot \nabla V \leq 0$. We simply write
	\begin{align*}
	\eqref{term-5} = -2 \iint |u(t,y)|^2 \psi_R(x-y) (x-y)\cdot \frac{x}{|x|} \partial_rV |u(t,x)|^2 dx dy. 
	\end{align*}
	
	{\bf Case 2:} $V$ is non-radially symmetric, $x\cdot \nabla V\leq 0$ and $\nabla^2 V$ is non-positive definite. We write
	\begin{align*}
	\eqref{term-5} &= - 2\iint |u(t,y)|^2 \psi_R(x-y) (x-y) \cdot \nabla V(x-y) |u(t,x)|^2 dx dy \nonumber \\
	& \mathrel{\phantom{=}} - 2\iint |u(t,y)|^2 \psi_R(x-y) (x-y) \cdot [\nabla V(x) - \nabla V(x-y)] |u(t,x)|^2 dx dy \nonumber \\
	&= - 2\iint |u(t,y)|^2 \psi_R(x-y) (x-y) \cdot \nabla V(x-y) |u(t,x)|^2 dx dy \nonumber \\
	& \mathrel{\phantom{=}} - 2\iint |u(t,y)|^2 \psi_R(x-y) \int_0^1 (x-y)  \nabla^2 V(x-y+\theta y) (x-y)^T d\theta |u(t,x)|^2 dx dy \nonumber \\
	&\geq 0. 
	\end{align*}
	We get
	\begin{align*}
	\frac{d}{dt} \Mca^{\otimes 2}_R(t) \geq &\frac{6\alpha}{\alpha+2} \iint |u(t,y)|^2 (\phi_R-\phi_{1,R})(x-y) |u(t,x)|^{\alpha+2} dx dy \\
	&+\frac{4\alpha}{\alpha+2} \iint |u(t,y)|^2 (\psi_R-\phi_R)(x-y) |u(t,x)|^{\alpha+2} dx dy \\
	&+\frac{6\alpha}{(\alpha+2)\omega_3 R^3} \iiint |\chi_R(y-z) u(t,y)|^2 |\chi_R(x-z) u(t,x)|^{\alpha+2} dxdydz \\
	&+ \iint |u(t,y)|^2 \nabla \left[ 3\phi_R(x-y) + 2(\psi_R-\phi_R)(x-y)\right] \cdot \nabla (|u(t,x)|^2) dx dy \\
	&+\frac{4}{\omega_3 R^3} \iiint |\chi_R(y-z) u(t,y)|^2 |\chi_R(x-z) \nabla u^\xi(t,x)|^2 dx dy dz \\
	&-2 \iint |u(t,y)|^2 \psi_R(x-y) (x-y)\cdot \frac{x}{|x|} \partial_rV |u(t,x)|^2 dx dy.
	\end{align*}
	By the defocusing nature and the fact $\psi_R-\phi_R \geq 0$, we infer that
	\begin{align*}
	\frac{4}{\omega_3R^3} \iiint &|\chi_R(y-z) u(t,y)|^2  |\chi_R(x-z)\nabla u^\xi(t,x)|^2 dxdydz  \\
	& \leq \frac{d}{dt} \Mca^{\otimes 2}_R(t)  - \frac{6\alpha}{\alpha+2} \iint |u(t,y)|^2 (\phi_R-\phi_{1,R})(x-y) |u(t,x)|^{\alpha+2} dx dy \nonumber \\
	&\mathrel{\phantom{\leq \frac{d}{dt} \Mca^{\otimes 2}_R(t)}} - \iint |u(t,y)|^2 \nabla \left[3\phi_R(x-y) + 2 (\psi_R-\phi_R)(x-y) \right] \cdot \nabla (|u(t,x)|^2) dx dy \nonumber \\
	&\mathrel{\phantom{\leq \frac{d}{dt} \Mca^{\otimes 2}_R(t)}} + 2\iint |u(t,y)|^2 \psi_R(x-y) (x-y)\cdot \frac{x}{|x|} \partial_rV |u(t,x)|^2 dx dy. 
	\end{align*}
	Here we use the convention $\partial_r V=0$ if $V$ is non-radially symmetric. The rest follows by the same argument as in the proof of Proposition \ref{prop-inter-mora-est-focus} using Corollary \ref{coro-mora-est-defocus} instead of Lemma \ref{lem-mora-est-focus}. The proof is complete.
	\end{proof}

	\section{Scattering criterion}
	\label{S5}
	\setcounter{equation}{0}
	
	\noindent {\it Proof of Theorem $\ref{theo-scat-crite}$.}
	We only give the proof in the focusing case. The one in the defocusing case is similar.	Our purpose is to check the scattering criteria given in Lemma $\ref{lem-scat-crite}$. To this end, we fix $a \in \R$ and let $\vareps>0$ sufficiently small and $T_0>0$ sufficiently large to be determined later. We will show that there exists $t_0 \in (a,a+T_0)$ such that $[t_0-\vareps^{-\sigma},t_0] \subset (a,a+T_0)$ and
	\begin{align} \label{scat-crite-check}
	\|u\|_{L^q([t_0-\vareps^{-\sigma},t_0]\times \R^3)} \lesssim \vareps^\mu
	\end{align}
	for some $\sigma, \mu>0$ satisfying \eqref{cond-sigma-mu}, where $q$ is as in \eqref{defi-q}. By \eqref{inter-mora-est}, there exist $T_0 =T_0(\vareps), J=J(\vareps)$, $R_0=R_0(\vareps,u_0,Q)$ and $\eta=\eta(\vareps)$ such that
	\[
	\frac{1}{JT_0} \int_a^{a+T_0} \int_{R_0}^{R_0e^J} \frac{1}{R^3} \iiint |\chi_R(y-z) u(t,y)|^2 |\nabla[\chi_R(x-z) u^{\xi}(t,x)]|^2 dxdydz \frac{dR}{R} dt \lesssim \vareps.
	\]
	It follows that there exists $R_1 \in [R_0,R_0e^J]$ such that
	\[
	\frac{1}{T_0} \int_a^{a+T_0} \frac{1}{R_1^3}\iiint |\chi_{R_1}(y-z)u(t,y)|^2 |\nabla[\chi_{R_1}(x-z) u^{\xi}(t,x)]|^2 dx dydz dt \lesssim \vareps
	\]
	hence
	\[
	\frac{1}{T_0} \int_a^{a+T_0} \frac{1}{R_1^3} \int \|\chi_{R_1}(\cdot-z) u(t)\|^2_{L^2} \|\nabla[\chi_{R_1}(\cdot-z) u^{\xi}(t)]\|^2_{L^2} dz dt \lesssim \vareps.
	\]
	By the change of variable $z=\frac{R_1}{4}(w+\theta)$ with $w\in \Z^3$ and $\theta \in [0,1]^3$, we deduce that there exists $\theta_1 \in [0,1]^3$ such that
	\[
	\frac{1}{T_0} \int_a^{a+T_0} \sum_{w\in \Z^3} \Big\|\chi_{R_1}\Big(\cdot-\frac{R_1}{4}(w+\theta_1)\Big) u(t)\Big\|^2_{L^2} \Big\|\nabla \Big[ \chi_{R_1}\Big(\cdot - \frac{R_1}{4}(w+\theta_1)\Big) u^{\xi}(t)\Big] \Big\|^2_{L^2} dt \lesssim \vareps.
	\]
	Let $\sigma>0$ to be chosen later. By dividing the interval $\left[a+\frac{T_0}{2}, a+\frac{3T_0}{4}\right]$ into $T_0\vareps^{\sigma}$ intervals of length $\vareps^{-\sigma}$, we infer that there exists $t_0 \in \left[a,\frac{T_0}{2},a+\frac{3T_0}{4}\right]$ such that $[t_0-\vareps^{-\sigma},t_0] \subset (a,a+T_0)$  and
	\[
	\int_{t_0-\vareps^{-\sigma}}^{t_0} \sum_{w\in \Z^3} \Big\|\chi_{R_1} \Big(\cdot -\frac{R_1}{4}(w+\theta_1)\Big)u(t)\Big\|^2_{L^2} \Big\|\nabla \Big[ \chi_{R_1}\Big(\cdot-\frac{R_1}{4}(w+\theta_1)\Big) u^{\xi}(t)\Big\|^2_{L^2} dt \lesssim \vareps^{1-\sigma}.
	\]
	This together with the Gagliardo-Nirenberg inequality
	\[
	\|u\|^4_{L^3} \lesssim \|u\|^2_{L^2} \|\nabla u^{\xi}\|^2_{L^2}
	\]
	imply that
	\begin{align} \label{est-t0-1}
	\int_{t_0 -\vareps^{-\sigma}}^{t_0} \sum_{w\in \Z^3} \Big\| \chi_{R_1} \Big( \cdot-\frac{R_1}{4}(w+\theta_1)\Big) u(t)\Big\|^4_{L^3} dt \lesssim \vareps^{1-\sigma}.
	\end{align}
	On the other hand, by H\"older's inequality, Cauchy-Schwarz inequality and Sobolev embedding, we have
	\begin{align}
	\sum_{w\in \Z^3} \Big\| \chi_{R_1} &\Big( \cdot-\frac{R_1}{4}(w+\theta_1) \Big) u(t)\Big\|^2_{L^3} \nonumber  \\
	&\leq \sum_{w\in \Z^3} \Big\|\chi_{R_1}\Big(\cdot-\frac{R_1}{4}(w+\theta_1)\Big) u(t)\Big\|_{L^2} \Big\|\chi_{R_1}\Big(\cdot-\frac{R_1}{4}(w+\theta_1)\Big) u(t)\Big\|_{L^6} \nonumber \\
	&\leq \Big(\sum_{w\in \Z^3} \Big\|\chi_{R_1}\Big(\cdot-\frac{R_1}{4}(w+\theta_1)\Big) u(t)\Big\|^2_{L^2}  \Big)^{1/2} \Big( \sum_{w\in \Z^3} \Big\|\chi_{R_1}\Big(\cdot-\frac{R_1}{4}(w+\theta_1)\Big) u(t)\Big\|^2_{L^6}\Big)^{1/2} \nonumber \\
	&\lesssim \|u(t)\|_{L^2} \|\nabla u(t)\|_{L^2} \lesssim 1. \label{est-t0-2}
	\end{align}
	Combining \eqref{est-t0-1} and \eqref{est-t0-2}, we get from the property of $\chi_{R_1}$ that
	\begin{align*}
	\|u\|^3_{L^3([t_0-\vareps^{-\sigma},t_0] \times \R^3)} &\lesssim \int_{t_0-\vareps^{-\sigma}}^{t_0} \sum_{w\in \Z^3} \Big\| \chi_{R_1} \Big( \cdot-\frac{R_1}{4} (w+\theta_1) \Big) u(t)\Big\|^3_{L^3} dt \\
	&\lesssim \int_{t_0-\vareps^{-\sigma}}^{t_0} \Big(\sum_{w\in \Z^3} \Big\|\chi_{R_1} \Big(\cdot- \frac{R_1}{4}(w+\theta_1) \Big) u(t)\Big\|^4_{L^3} \Big)^{\frac{1}{2}} \\
	&\mathrel{\phantom{\lesssim \int_{t_0-\vareps^{-\sigma}}^{t_0}}} \times \Big( \sum_{w\in \Z^3} \Big\|\chi_{R_1} \Big(\cdot- \frac{R_1}{4}(w+\theta_1) \Big) u(t)\Big\|^2_{L^{3}}\Big)^{\frac{1}{2}} dt  \\
	&\lesssim \Big(\int_{t_0-\vareps^{-\sigma}}^{t_0} \sum_{w\in \Z^3} \Big\| \chi_{R_1} \Big( \cdot- \frac{R_1}{4} (w+\theta_1)\Big) u(t)\Big\|^4_{L^{3}} dt \Big)^{\frac{1}{2}} \\
	&\mathrel{\phantom{\lesssim \int_{t_0-\vareps^{-\sigma}}^{t_0}}} \times \Big(\int_{t_0-\vareps^{-\sigma}}^{t_0} \sum_{w\in \Z^3} \Big\| \chi_{R_1} \Big( \cdot- \frac{R_1}{4} (w+\theta_1)\Big) u(t)\Big\|^2_{L^{3}} dt \Big)^{\frac{1}{2}} \\
	&\lesssim \vareps^{\frac{1-\sigma}{2}} \vareps^{-\frac{\sigma}{2}} = \vareps^{\frac{1}{2} -\sigma} 
	\end{align*}
	which implies that
	\begin{align} \label{est-t0}
	\|u\|_{L^3([t_0-\vareps^{-\sigma},t_0] \times \R^3)} \lesssim \vareps^{\frac{1}{3}\left(\frac{1}{2}-\sigma \right)}.
	\end{align}
	By interpolation, we have
	\begin{align*}
	\|u\|_{L^q([t_0-\vareps^{-\sigma},t_0]\times \R^3)} &\leq \|u\|^{\vartheta}_{L^3([t_0-\vareps^{-\sigma},t_0]\times \R^3)} \|u\|^{1-\vartheta}_{L^{10}([t_0-\vareps^{-\sigma},t_0]\times \R^3)} \\
	&\lesssim \vareps^{\frac{\vartheta}{3}\left(\frac{1}{2}-\sigma\right)} \vareps^{-\frac{\sigma}{10}(1-\vartheta)} \\
	&= \vareps^{\frac{\vartheta}{6} - \sigma \left(\frac{1}{10}+\frac{7\vartheta}{30}\right)},
	\end{align*}
	where
	\[
	\vartheta=\frac{3(4-\alpha)}{7\alpha} \in (0,1).
	\]
	Here we have used the fact that
	\[
	\|u\|_{L^{10}(I\times \R^3)} \lesssim \scal{I}^{\frac{1}{10}}
	\]
	which follows from the local theory. This shows \eqref{scat-crite-check} with
	\[
	\mu=\frac{\vartheta}{6} - \sigma \left(\frac{1}{10}+\frac{7\vartheta}{30}\right) = \frac{4-\alpha}{14\alpha} -\frac{2\sigma}{5\alpha}.
	\]
	By taking $0<\sigma<\frac{4-\alpha}{7}$, we see that \eqref{cond-sigma-mu} is satisfied. The proof is complete.
	\hfill $\Box$
	
	\section{Long time dynamics}
	\label{S6}
	\setcounter{equation}{0}
	In this section, we give the proofs of Theorem \ref{theo-scat-below} and Theorem \ref{theo-scat-at}.
	
	\noindent {\it Proof of Theorem \ref{theo-scat-below}.}
	Thanks to Theorem \ref{theo-scat-crite}, it suffices to show \eqref{est-solu-below}. To see this, we first claim that there exists $\rho=\rho(u_0,Q)>0$ such that
	\begin{align} \label{claim-below}
	\|\nabla u(t)\|_{L^2} \|u(t)\|^{\sigc}_{L^2} \leq (1-\rho) \|\nabla Q\|_{L^2} \|Q\|^{\sigc}_{L^2}
	\end{align}
	for all $t\in (-T_*,T^*)$. Assume \eqref{claim-below} for the moment, let us prove \eqref{est-solu-below}. By the Gagliardo-Nirenberg inequality \eqref{GN-ineq} and \eqref{claim-below}, we have
	\begin{align*}
	\|u(t)\|^{\alpha+2}_{L^{\alpha+2}} \|u(t)\|^{2\sigc}_{L^2} &\leq C_{\opt} \|\nabla u(t)\|^{\frac{3\alpha}{2}}_{L^2} \|u(t)\|^{\frac{4-\alpha}{2} +2\sigc}_{L^2} \\
	&= C_{\opt} \left(\|\nabla u(t)\|_{L^2} \|u(t)\|^{\sigc}_{L^2} \right)^{\frac{3\alpha}{2}} \\
	&\leq C_{\opt} (1-\rho)^{\frac{3\alpha}{2}} \left(\|\nabla Q\|_{L^2} \|Q\|^{\sigc}_{L^2}\right)^{\frac{3\alpha}{2}}
	\end{align*}
	for all $t\in (-T_*,T^*)$. From this, \eqref{poho-iden} and \eqref{sharp-const-GN}, we see that
	\[
	\|u(t)\|^{\alpha+2}_{L^{\alpha+2}} \|u(t)\|^{2\sigc}_{L^2} \leq  \frac{2(\alpha+2)}{3\alpha} (1-\rho)^{\frac{3\alpha}{2}}\|\nabla Q\|_{L^2}^2 \|Q\|^{2\sigc}_{L^2} = (1-\rho)^{\frac{3\alpha}{2}} \|Q\|^{\alpha+2}_{L^{\alpha+2}} \|Q\|^{2\sigc}_{L^2}
	\]
	for all $t \in (-T_*,T^*)$ which proves \eqref{est-solu-below}. 
	
	Now, we prove \eqref{claim-below}. Recall that the initial data is assumed to satisfy \eqref{cond-1-HI} and \eqref{cond-2-HI}. To this end, we multiply both sides of $E(u(t))$ with $[M(u(t))]^{\sigc}$ and use the Gagliardo-Nirenberg inequality \eqref{GN-ineq} together with $V\geq 0$ to have
		\begin{align}
		E(u(t)) [M(u(t))]^{\sigc} & = \left(\frac{1}{2} \|\nabla u(t)\|^2_{L^2} +\frac{1}{2} \int V|u(t)|^2 dx - \frac{1}{\alpha+2} \|u(t)\|^{\alpha+2}_{L^{\alpha+2}} \right) \|u(t)\|^{2\sigc}_{L^2} \nonumber \\
		&\geq \frac{1}{2} \left( \|\nabla u(t)\|_{L^2} \|u(t)\|^{\sigc}_{L^2}\right)^2 - \frac{C_{\opt}}{\alpha+2} \|\nabla u(t)\|^{\frac{3\alpha}{2}}_{L^2} \|u(t)\|^{\frac{4-\alpha}{2}+2\sigc}_{L^2} \nonumber \\
		& = G\left(\|\nabla u(t)\|_{L^2} \|u(t)\|^{\sigc}_{L^2} \right), \label{est-E}
		\end{align}
		where
		\begin{align} \label{defi-G}
		G(\lambda) = \frac{1}{2} \lambda^2 - \frac{C_{\opt}}{\alpha+2} \lambda^{\frac{3\alpha}{2}}.
		\end{align}
		Using \eqref{iden-Q}, we see that
		\[
		G \left( \|\nabla Q\|_{L^2} \|Q\|^{\sigc}_{L^2} \right) = \frac{3\alpha-4}{6\alpha} \left( \|\nabla Q\|_{L^2} \|Q\|^{\sigc}_{L^2}  \right)^2 = E_0(Q) [M(Q)]^{\sigc}.
		\]
		From \eqref{cond-1-HI}, \eqref{est-E}, the conservation of mass and energy, we have
		\[
		G\left(\|\nabla u(t)\|_{L^2} \|u(t)\|^{\sigc}_{L^2} \right) \leq E(u_0) [M(u_0)]^{\sigc} < E_0(Q) [M(Q)]^{\sigc} = G \left( \|\nabla Q\|_{L^2} \|Q\|^{\sigc}_{L^2} \right)
		\]
		for all $t$ in the existence time. By \eqref{cond-2-HI}, the continuity argument implies 
		\begin{align} \label{est-solu-1}
		\|\nabla u(t)\|_{L^2} \|u(t)\|^{\sigc}_{L^2} < \|\nabla Q\|_{L^2} \|Q\|^{\sigc}_{L^2}
		\end{align}
		for all $t\in (-T_*,T^*)$. Next, from \eqref{cond-1-HI}, we take $\vartheta = \vartheta(u_0,Q)>0$ such that
		\begin{align} \label{cond-1-app}
		E(u_0) [M(u_0)]^{\sigc} \leq (1-\vartheta) E_0(Q) [M(Q)]^{\sigc}.
		\end{align}
		Using the fact that
		\[
		E_0(Q)[M(Q)]^{\sigc}=\frac{3\alpha-4}{6\alpha} \left(\|\nabla Q\|_{L^2} \|Q\|^{\sigc}_{L^2} \right)^2 = \frac{3\alpha-4}{4(\alpha+2)} C_{\opt} \left(\|\nabla Q\|_{L^2} \|Q\|^{\sigc}_{L^2} \right)^{\frac{3\alpha}{2}},
		\]
		we infer from \eqref{est-E} and \eqref{cond-1-app} that
		\begin{align} \label{est-H}
		\frac{3\alpha}{3\alpha-4} \left(\frac{\|\nabla u(t)\|_{L^2} \|u(t)\|^{\sigc}_{L^2}}{\|\nabla Q\|_{L^2} \|Q\|^{\sigc}_{L^2}} \right)^2  - \frac{4}{3\alpha-4} \left(\frac{\|\nabla u(t)\|_{L^2} \|u(t)\|^{\sigc}_{L^2}}{\|\nabla Q\|_{L^2} \|Q\|^{\sigc}_{L^2}} \right)^{\frac{3\alpha}{2}} \leq 1-\vartheta.
		\end{align}
		We consider the function $H(\lambda) = \frac{3\alpha}{3\alpha-4} \lambda^2 - \frac{4}{3\alpha-4} \lambda^{\frac{3\alpha}{2}}$ with $0<\lambda<1$ due to \eqref{est-solu-1}. We see that $H$ is strictly increasing in $(0,1)$ with $H(0) = 0$ and $H(1) = 1$. It follows from \eqref{est-H} that there exists $\rho>0$ depending on $\vartheta$ such that $\lambda \leq 1-\rho$ which shows \eqref{claim-below}. The proof is complete.
	\hfill $\Box$ 
	
	We next study the long time dynamics at the ground state threshold given in Theorem \ref{theo-scat-at}. 
	
	\noindent {\it Proof of Theorem \ref{theo-scat-at}.}
	Let us start with the following observation. There is no $f\in H^1$ satisfying
	\begin{align} \label{observation}
	E(f) [M(f)]^{\sigc} = E_0(Q) [M(Q)]^{\sigc}, \quad \|\nabla f\|_{L^2} \|f\|^{\sigc}_{L^2} = \|\nabla Q\|_{L^2} \|Q\|^{\sigc}_{L^2}.
	\end{align}
	In fact, we take $\lambda>0$ such that $\|f\|_{L^2} = \lambda \|Q\|_{L^2}$. It follows that
	\begin{align} \label{obser-f}
	E(f) = \lambda^{-2\sigc} E_0(Q), \quad \|\nabla f\|_{L^2} = \lambda^{-\sigc} \|\nabla Q\|_{L^2}.
	\end{align}
	Using the Gagliardo-Nirenberg inequality \eqref{GN-ineq} and \eqref{poho-iden}, we see that
	\begin{align*}
	\|f\|^{\alpha+2}_{L^{\alpha+2}} \|f\|^{2\sigc}_{L^2} &\leq C_{\opt} \|\nabla f\|^{\frac{3\alpha}{2}}_{L^2} \|f\|^{\frac{4-\alpha}{2} + 2\sigc}_{L^2} \\
	& = \frac{2(\alpha+2)}{3\alpha} \left(\|\nabla Q\|_{L^2} \|Q\|^{\sigc}_{L^2}\right)^{-\frac{3\alpha-4}{2}}  \left(\|\nabla f\|_{L^2}\|f\|^{\sigc}_{L^2}\right)^{\frac{3\alpha}{2}} \\
	& = \frac{2(\alpha+2)}{3\alpha} \left(\|\nabla Q\|_{L^2}\|Q\|_{L^2}^{\sigc}\right)^2.
	\end{align*}
	This implies
	\[
	\|f\|^{\alpha+2}_{L^{\alpha+2}} \leq \frac{2(\alpha+2)}{3\alpha} \lambda^{-2\sigc} \|\nabla Q\|^2_{L^2} = \lambda^{-2\sigc} \|Q\|^{\alpha+2}_{L^{\alpha+2}}.
	\]
	Using \eqref{obser-f}, we infer that
	\begin{align*}
	0\leq \int_{\R^3} V(x)|f(x)|^2 dx = \frac{1}{\alpha+2}\|f\|^{\alpha+2}_{L^{\alpha+2}} - \frac{1}{\alpha+2} \lambda^{-2\sigc} \|Q\|^{\alpha+2}_{L^{\alpha+2}} \leq 0.
	\end{align*}
	This shows that $f \equiv 0$ which is impossible. 
	
	Now, let $u_0\in H^1$ satisfy \eqref{cond-enegy-at} and \eqref{cond-grad-at}. Let $u:(-T_*,T^*)\times \R^3 \rightarrow \C$ be the corresponding solution to the focusing problem \eqref{NLS-V}. By \eqref{est-E}, we have
	\begin{align} \label{proper-G}
	G\left(\|\nabla u(t)\|_{L^2} \|u(t)\|^{\sigc}_{L^2} \right) \leq E(u(t)) [M(u(t))]^{\sigc} = E(u_0) [M(u_0)]^{\sigc} = E_0(Q) [M(Q)]^{\sigc}
	\end{align}
	for all $t\in (-T_*,T^*)$, where $G$ is as in \eqref{defi-G}. It is easy to check that $G$ attains its maximum at 
	\[
	\lambda_0 = \left(\frac{2(\alpha+2)}{3\alpha C_{\opt}}\right)^{\frac{2}{3\alpha-4}} = \|\nabla Q\|_{L^2} \|Q\|^{\sigc}_{L^2}
	\]
	and 
	\[
	G(\lambda_0) = E_0(Q) [M(Q)]^{\sigc}.
	\]
	We claim that 
	\begin{align} \label{claim-at}
	\|\nabla u(t)\|_{L^2} \|u(t)\|^{\sigc}_{L^2} < \|\nabla Q\|_{L^2} \|Q\|^{\sigc}_{L^2}
	\end{align}
	for all $t\in (-T_*,T^*)$. By the conservation of mass and the local theory, we have $T_*=T^*=\infty$, i.e. the solution exists globally in time. We will prove \eqref{claim-at} by contradiction. Suppose that it is not true. Then there exists $t_0 \in (-T_*,T^*)$ such that 
	\[
	\|\nabla u(t_0)\|_{L^2} \|u(t_0)\|^{\sigc}_{L^2} \geq \|\nabla Q\|_{L^2} \|Q\|^{\sigc}_{L^2}.
	\] 
	By continuity using \eqref{cond-enegy-at}, there exists $t_1\in (-T_*,T^*)$ such that 
	\[
	\|\nabla u(t_1)\|_{L^2} \|u(t_1)\|^{\sigc}_{L^2} = \|\nabla Q\|_{L^2} \|Q\|^{\sigc}_{L^2}.
	\]
	Thanks to \eqref{cond-enegy-at} and the conservation of mass and energy, we have
	\[
	E(u(t_1)) [M(u(t_1))]^{\sigc} = E_0(Q) [M(Q)]^{\sigc}
	\]
	which contradicts the observation \eqref{observation}.
	
	By \eqref{claim-at}, we consider two cases.
	
	{\bf Case 1.} If 
	\[
	\sup_{[0,\infty)} \|\nabla u(t)\|_{L^2} \|u(t)\|^{\sigc}_{L^2} < \|\nabla Q\|_{L^2} \|Q\|^{\sigc}_{L^2},
	\] 
	then there exists $\rho>0$ such that
	\[
	\|\nabla u(t)\|_{L^2} \|u(t)\|^{\sigc}_{L^2} \leq (1-\rho) \|\nabla Q\|_{L^2} \|Q\|^{\sigc}_{L^2}
	\]
	for all $t\in [0,\infty)$. By the same argument as in the proof of Theorem \ref{theo-scat-below}, we prove \eqref{scat-cond}. By Theorem \ref{theo-scat-crite}, the corresponding solution scatters in $H^1$ forward in time.
	
	{\bf Case 2.} If 
	\[
	\sup_{t\in [0,\infty)} \|\nabla u(t)\|_{L^2} \|u(t)\|_{L^2}^{\sigc} = \|\nabla Q\|_{L^2} \|Q\|^{\sigc}_{L^2},
	\]
	then there exists $(t_n)_{n\geq 1} \subset [0,\infty)$ such that 
	\[
	\lim_{n\rightarrow \infty} \|\nabla u(t_n)\|_{L^2} \|u(t_n)\|_{L^2}^{\sigc} = \|\nabla Q\|_{L^2} \|Q\|^{\sigc}_{L^2}.
	\]
	By \eqref{cond-enegy-at} and the conservation laws of mass and energy, we have
	\[
	E(u(t_n))[M(u(t_n))]^{\sigc}= E_0(Q) [M(Q)]^{\sigc}.
	\]
	Note that $t_n$ must tend to infinity. Otherwise, there exists $t_0 \in [0,\infty)$ such that up to a subsequence, $t_n \rightarrow t_0$ as $n\rightarrow \infty$. By continuity, we have
	\[
	E(u(t_0))[M(u(t_0))]^{\sigc}= E_0(Q) [M(Q)]^{\sigc}, \quad \|\nabla u(t_0)\|_{L^2} \|u(t_0)\|_{L^2}^{\sigc} = \|\nabla Q\|_{L^2} \|Q\|^{\sigc}_{L^2} 
	\]
	which is impossible due to the observation \eqref{observation}. Now, we take $\lambda>0$ so that $\|u(t_n)\|_{L^2} = \lambda \|Q\|_{L^2}$. Note that $\lambda$ is independent of $n$ due to the conservation of mass. It follows that
	\[
	E(u(t_n)) = \lambda^{-2\sigc} E_0(Q), \quad \lim_{n\rightarrow \infty} \|\nabla u(t_n)\|_{L^2} = \lambda^{-\sigc} \|\nabla Q\|_{L^2}.
	\]
	By the Gagliardo-Nirenberg inequality \eqref{GN-ineq}, we see that
	\begin{align*}
	\|u(t_n)\|^{\alpha+2}_{L^{\alpha+2}} &\leq C_{\opt} \|\nabla u(t_n)\|^{\frac{3\alpha}{2}}_{L^2} \|u(t_n)\|^{\frac{4-\alpha}{2}}_{L^2} \\
	&= \frac{2(\alpha+2)}{3\alpha} \left(\|\nabla Q\|_{L^2} \|Q\|^{\sigc}_{L^2}\right)^{-\frac{3\alpha-4}{2}} \|\nabla u(t_n)\|^{\frac{3\alpha}{2}}_{L^2} \left(\lambda \|Q\|_{L^2}\right)^{\frac{4-\alpha}{2}}
	\end{align*}
	which implies
	\[
	\lim_{n\rightarrow \infty} \|u(t_n)\|^{\alpha+2}_{L^{\alpha+2}} \leq \frac{2(\alpha+2)}{3\alpha} \lambda^{-2\sigc} \|\nabla Q\|^2_{L^2} = \lambda^{-2\sigc} \|Q\|^{\alpha+2}_{L^{\alpha+2}}.
	\]
	Thus, we have
	\[
	\lambda^{-2\sigc} E_0(Q) \leq \lim_{n\rightarrow \infty} E_0(u(t_n)) \leq E(u(t_n)) = \lambda^{-2\sigc} E_0(Q)
	\]
	which implies
	\[
	\lim_{n\rightarrow \infty} E_0(u(t_n)) = \lambda^{-2\sigc} E_0(Q).
	\]
	We also have that
	\begin{align} \label{limi-V}
	\lim_{n\rightarrow \infty} \int_{\R^3} V(x)|u(t_n,x)|^2 dx =0.
	\end{align}
	We have proved that there exists a time sequence $t_n \rightarrow \infty$ such that 
	\[
	\|u(t_n)\|_{L^2} = \lambda \|Q\|_{L^2}, \quad \lim_{n\rightarrow \infty} \|\nabla u(t_n)\|_{L^2} = \lambda^{-\sigc} \|\nabla Q\|_{L^2}, \quad \lim_{n\rightarrow\infty} E_0(u(t_n))= \lambda^{-2\sigc} E_0(Q)
	\]
	for some $\lambda>0$. By the concentration-compactness lemma of Lions \cite{Lions}, there exists a subsequence still denoted by $(u(t_n))_{n\geq 1}$ satisfying one of the following three possibilities: vanishing, dichotomy and compactness. 
	
	The vanishing cannot occur. In fact, suppose that the vanishing occurs. Then it was shown in \cite{Lions} that $u(t_n) \rightarrow 0$ strongly in $L^r$ for any $2<r<6$. This however contradicts to the fact that
	\[
	\lim_{n\rightarrow \infty} \|u(t_n)\|^{\alpha+2}_{L^{\alpha+2}} = \lambda^{-2\sigc} \|Q\|^{\alpha+2}_{L^{\alpha+2}}.
	\]
	
	The dichotomy cannot occur. Indeed, suppose the dichotomy occurs, then there exist $\mu \in (0, \lambda \|Q\|_{L^2})$ and sequences $(f^1_n)_{n\geq 1}, (f^2_n)_{n\geq 1}$ bounded in $H^1$ such that
	\[
	\left\{
	\renewcommand*{\arraystretch}{1.3}
	\begin{array}{l}
	\|u(t_n) - f^1_n - f^2_n\|_{L^r} \rightarrow 0 \text{ as } n\rightarrow \infty \text{ for any } 2 \leq r < 6, \\
	\|f^1_n\|_{L^2} \rightarrow \mu, \quad \|f^2_n\|_{L^2} \rightarrow \lambda \|Q\|_{L^2} - \mu \text{ as } n\rightarrow \infty, \\
	\dist(\supp(f^1_n), \supp(f^2_n)) \rightarrow \infty \text{ as } n \rightarrow \infty, \\
	\liminf_{n\rightarrow \infty} \|\nabla u(t_n)\|^2_{L^2} - \|\nabla f^1_n\|^2_{L^2} - \|\nabla f^2_n\|^2_{L^2} \geq 0.
	\end{array}
	\right.	
	\]
	By the Gagliardo-Nirenberg inequality, we have
	\[
	\|f^1_n\|^{\alpha+2}_{L^{\alpha+2}} \leq C_{\opt} \|\nabla f^1_n\|^{\frac{3\alpha}{2}}_{L^2} \|f^1_n\|^{\frac{4-\alpha}{2}}_{L^2}
	\]
	which implies
	\[
	\lim_{n\rightarrow \infty} \|f^1_n\|^{\alpha+2}_{L^{\alpha+2}} < C_{\opt} \lim_{n\rightarrow \infty} \|\nabla f^1_n\|^{\frac{3\alpha}{2}}_{L^2} \|u(t_n)\|^{\frac{4-\alpha}{2}}_{L^2}.
	\]
	A similar estimate holds for $f^2_n$. It follows that
	\begin{align*}
	\lambda^{-2\sigc} \|Q\|^{\alpha+2}_{L^{\alpha+2}}= \lim_{n\rightarrow \infty} \|u(t_n)\|^{\alpha+2}_{L^{\alpha+2}} &= \lim_{n\rightarrow \infty} \|f^1_n\|^{\alpha+2}_{L^{\alpha+2}} + \|f^2_n\|^{\alpha+2}_{L^{\alpha+2}} \\
	&< C_{\opt} \lim_{n\rightarrow \infty} \left(\|\nabla f^1_n\|^{\frac{3\alpha}{2}}_{L^2} + \|\nabla f^2_n\|^{\frac{3\alpha}{2}}_{L^2} \right) \|u(t_n)\|^{\frac{4-\alpha}{2}}_{L^2} \\
	&\leq C_{\opt} \lim_{n\rightarrow \infty} \left(\|\nabla f^1_n\|_{L^2}^2 +\|\nabla f^2_n\|^2 \right)^{\frac{3\alpha}{4}} \|u(t_n)\|^{\frac{4-\alpha}{2}}_{L^2}\\
	&\leq C_{\opt} \lim_{n\rightarrow \infty} \|\nabla u(t_n)\|^{\frac{3\alpha}{2}}_{L^2} \|u(t_n)\|^{\frac{4-\alpha}{2}}_{L^2} \\
	&= C_{\opt} \left(\lambda^{-\sigc} \|\nabla Q\|_{L^2} \right)^{\frac{3\alpha}{2}} \left( \lambda \|Q\|_{L^2}\right)^{\frac{4-\alpha}{2}} \\
	&= \lambda^{-2\sigc} \|Q\|^{\alpha+2}_{L^{\alpha+2}}
	\end{align*}
	which is a contradiction. 
	
	Therefore, the compactness must occur. By \cite{Lions}, there exist a subsequence still denoted by $(u(t_n))_{n\geq 1}$, a function $f \in H^1$ and a sequence $(y_n)_{n\geq 1} \subset \R^N$ such that $u(t_n,\cdot+y_n) \rightarrow f$ strongly in $L^r$ for any $2\leq r<6$ and weakly in $H^1$. We have
	\[
	\|f\|_{L^2} = \lim_{n\rightarrow \infty} \|u(t_n, \cdot+y_n)\|_{L^2} = \lambda \|Q\|_{L^2}
	\]
	and
	\[
	\|f\|^{\alpha+2}_{L^{\alpha+2}} = \lim_{n\rightarrow \infty} \|u(t_n, \cdot+y_n)\|^{\alpha+2}_{L^{\alpha+2}} = \lambda^{-2\sigc}\|Q\|^{\alpha+2}_{L^{\alpha+2}}
	\]
	and 
	\[
	\|\nabla f\|_{L^2} \leq \liminf_{n\rightarrow \infty} \|\nabla u(t_n, \cdot+y_n)\|_{L^2} = \lambda^{-\sigc} \|\nabla Q\|_{L^2}.
	\]
	On the other hand, by the Gagliardo-Nirenberg inequality, we have
	\[
	\|\nabla f\|^{\frac{3\alpha}{2}}_{L^2} \geq \frac{\|f\|^{\alpha+2}_{L^{\alpha+2}}}{C_{\opt} \|f\|^{\frac{4-\alpha}{2}}_{L^2}} = \frac{ \lambda^{-2\sigc} \|Q\|^{\alpha+2}_{L^{\alpha+2}}}{ C_{\opt} \left( \lambda \|Q\|_{L^2}\right)^{\frac{4-\alpha}{2}}} = \left(\lambda^{-\sigc} \|\nabla Q\|_{L^2}\right)^{\frac{3\alpha}{2}}
	\]
	hence $\|\nabla f\|_{L^2} = \lim_{n\rightarrow \infty} \|\nabla u(t_n,\cdot+y_n)\|_{L^2} = \lambda^{-\sigc}\|\nabla Q\|_{L^2}$. In particular, $u(t_n,\cdot+y_n) \rightarrow f$ strongly in $H^1$. It is easy to see that
	\[
	\frac{\|f\|^{\alpha+2}_{L^{\alpha+2}}}{\|\nabla f\|^{\frac{3\alpha}{2}}_{L^2} \|f\|^{\frac{4-\alpha}{2}}_{L^2}} = \frac{\|Q\|^{\alpha+2}_{L^{\alpha+2}}}{\|\nabla Q\|^{\frac{3\alpha}{2}}_{L^2} \|Q\|^{\frac{4-\alpha}{2}}_{L^2}} = C_{\opt}.
	\]
	This shows that $f$ is an optimizer for the Gagliardo-Nirenberg inequality \eqref{GN-ineq}. By the characterization of ground state (see e.g. \cite{Lions}) with the fact $\|f\|_{L^2} = \lambda \|Q\|_{L^2}$, we have $f(x) = e^{i\theta} \lambda Q(x - x_0)$ for some $\theta \in \R$, $\mu>0$ and $x_0 \in \R^N$. Redefining the variable, we prove that there exists a sequence $(y_n)_{n\geq 1} \subset \R^N$ such that
	\[
	u(t_n, \cdot +y_n) \rightarrow e^{i\theta} \lambda Q \text{ strongly in } H^1
	\]
	as $n\rightarrow \infty$. Finally, using \eqref{limi-V}, we infer that $|y_n| \rightarrow \infty$ as $n\rightarrow \infty$. In fact, suppose that $y_n \rightarrow y_0 \in \R^N$ as $n\rightarrow \infty$. We have from \eqref{limi-V} that
	\begin{align*}
	0 = \lim_{n\rightarrow\infty} \int_{\R^N} V(x) |u(t_n,x)|^2 dx &=\lim_{n\rightarrow \infty} \int_{\R^N} V(x+y_n) |u(t_n, x+y_n)|^2 dx \\
	&= \lambda^2 \int_{\R^N} V(x+y_0) |Q(x)|^2 dx 
	\end{align*}
	which is a contradiction. The proof is complete.
	\hfill $\Box$
	
	\section{Remark on long time dynamics for NLS with repulsive inverse-power potentials}
	\label{S7}
	\setcounter{equation}{0}
	
	Let us now consider the NLS with repulsive inverse-power potentials in three dimensions, namely 
	\begin{align} \label{NLS-inverse}
	\left\{ 
	\begin{array}{rcl}
	i\partial_t u +\Delta u - c|x|^{-\sigma} u &=& \pm |u|^\alpha u, \quad (t,x) \in \R \times \R^3, \\
	u(0,x)&=& u_0(x),
	\end{array}
	\right.
	\end{align}
	where $c>0$, $0<\sigma<2$ and $\alpha>0$. In the case $\sigma=1$, \eqref{NLS-inverse} becomes the well-known NLS with Coulomb potential. The local well-posedness, global well-posedness and finite time blow-up of $H^1$-solutions for \eqref{NLS-inverse} have been studied in \cite{MZZ, Dinh-repul}. It is known that $H^1$-solutions satisfy the conservation of mass and energy
	\begin{align*}
	M(u(t)) &= \int |u(t,x)|^2 dx = M(u_0), \\
	E(u(t)) &= \frac{1}{2} \int |\nabla u(t,x)|^2 dx +\frac{c}{2} \int |x|^{-\sigma} |u(t,x)|^2 dx \pm \frac{1}{\alpha+2} \int |u(t,x)|^{\alpha+2} dx = E(u_0).
	\end{align*}
	In the defocusing, thanks to global in time Strichartz estimates proved by Mizutani \cite{Mizutani}, the energy scattering for \eqref{NLS-inverse} was shown in \cite{MZZ, Dinh-repul}. In the focusing case, we can apply the argument presented in the paper (especially in the radial case) to show long time dynamics for \eqref{NLS-inverse}. More precisely, we have the following results. 
	
	\begin{theorem} [Scattering below the ground state threshold]\label{theo-NLS-inverse-below}
		Let $\frac{4}{3}<\alpha<4$, $c>0$ and $0<\sigma<2$. Let $u_0 \in H^1$ satisfy \eqref{cond-1-HI} and \eqref{cond-2-HI}. Then the corresponding solution to the focusing problem \eqref{NLS-inverse} exists globally in time and scatters in $H^1$ in both directions.
	\end{theorem}
	
	\begin{theorem} [Scattering at the ground state threshold] \label{theo-NLS-inverse-at}
		Let $\frac{4}{3}<\alpha<4$, $c>0$ and $0<\sigma<2$. Let $u_0 \in H^1$ satisfy
		\begin{align*} 
		E(u_0) [M(u_0)]^{\sigc} &= E_0(Q) [M(Q)]^{\sigc}, \\
		\|\nabla u_0\|_{L^2} \|u_0\|^{\sigc}_{L^2} &< \|\nabla Q\|_{L^2} \|Q\|^{\sigc}_{L^2}.
		\end{align*}
		Then the corresponding solution to the focusing problem \eqref{NLS-inverse} exists globally in time. Moreover, the solution either scatters in $H^1$ forward in time, or there exist a time sequence $t_n\rightarrow \infty$ and a sequence $(y_n)_{n\geq 1} \subset \R^3$ satisfying $|y_n| \rightarrow \infty$ such that
		\[
		u(t_n, \cdot+y_n) \rightarrow e^{i\theta} \lambda Q \quad \text{strongly in } H^1 
		\]
		for some $\theta \in \R$ and $\lambda:= \frac{\|u_0\|_{L^2}}{\|Q\|_{L^2}}$ as $n\rightarrow \infty$.
	\end{theorem}

	Note that dispersive estimates for \eqref{NLS-inverse} was proved by Goldberg \cite{Goldberg} in three dimensions, however, dispersive estimates for dimensions $N\geq 4$ are still unknown.

	\section*{Acknowledgement}
	This work was supported in part by the Labex CEMPI (ANR-11-LABX-0007-01). The author would like to express his deep gratitude to his wife - Uyen Cong for her encouragement and support. He would like to thank Takahisa Inui for pointing out Remark \ref{rem-zero-poten}. 
	
	\appendix
	
	\section*{Appendix}
	
	In this appendix, we will show Remark $\ref{rem-exam-V}$. Let $V$ be as in \eqref{exam-V}. We first compute
	\begin{align*}
	\|V\|_{L^q}  &= |c| \||x|^{-\sigma}e^{-a|x|}\|_{L^q}  \\
	&= |c| \left( \int_{\R^3} |x|^{-q\sigma} e^{-aq|x|} dx\right)^{\frac{1}{q}} \\
	&= |c| \left( 4\pi \int_0^\infty r^{2-q\sigma} e^{-aqr} dr \right)^{\frac{1}{q}} \\
	&= |c| \left[4\pi (aq)^{q\sigma-3} \Gamma(3-q\sigma)\right]^{\frac{1}{q}}
	\end{align*}
	which proves \eqref{norm-Lq-V}.
	
	We now compute 
	\[
	\|V\|_{\mathcal{K}} = \sup_{x\in \R^3} \int_{\R^3} \frac{|V(y)|}{|x-y|} dy.
	\]
	Consider
	\[
	\int \frac{|V(y)|}{|x-y|} dy = |c| \int \frac{e^{-a|y|}}{|y|^\sigma |x-y|} dy.
	\]
	In the case $x=0$, we have
	\[
	\int \frac{e^{-a|y|}}{|y|^{1+\sigma}} dy = 4\pi \int_0^\infty e^{-ar} r^{1-\sigma} dr = 4\pi a^{\sigma -2} \Gamma(2-\sigma).
	\]
	In the case $x \ne 0$, we write
	\begin{align*}
	\int \frac{e^{-a|y|}}{|y|^\sigma |x-y|} dy &= \int_0^\infty \int_{\Sb^2} \frac{e^{-ar}}{r^\sigma |x-r\theta|} r^2 dr d\theta \\
	&=\int_0^\infty e^{-ar} r^{1-\sigma} I(x,r) dr,
	\end{align*}
	where $r = |y|$ and
	\[
	I(x,r) = \int_{\Sb^2} \frac{1}{\left|\frac{x}{r} -\theta\right|} d\theta.
	\]
	Take $A \in O(3)$ such that $Ae_1 = \frac{x}{|x|}$ with $e_1=(1,0,0)$, we see that
	\[
	I(x,r) = \int_{\Sb^2} \frac{1}{\left| \frac{|x|}{r} Ae_1 - \theta \right|} d\theta = \int_{\Sb^2} \frac{1}{\left| \frac{|x|}{r} e_1 - \theta \right|} d\theta.
	\]
	By change of variables, we arrive
	\begin{align*}
	I(x,r) &= \int_{-1}^1 \int_{\sqrt{1-s^2} \Sb^1} \frac{d\eta}{\sqrt{\left(\frac{|x|}{r}-s \right)^2 +|\eta|^2}} \frac{ds}{\sqrt{1-s^2}} \\
	&= \int_{-1}^1 \int_{\Sb^1} \frac{\sqrt{1-s^2} d\zeta}{\sqrt{\left(\frac{|x|}{r}-s \right)^2 +1-s^2}} \frac{ds}{\sqrt{1-s^2}} \\
	&=|\Sb^1| \int_{-1}^1 \frac{ds}{\sqrt{\left(\frac{|x|}{r}-s \right)^2 +1-s^2}} \\
	&= 2\pi \frac{r}{|x|} \left( \frac{|x|}{r} +1 - \left| \frac{|x|}{r}-1 \right|\right) \\
	&= \left\{
	\begin{array}{cl}
	4\pi &\text{if } |x| \leq r, \\
	4\pi \frac{r}{|x|} &\text{if } |x|\geq r.
	\end{array}
	\right.
	\end{align*}
	It follows that
	\begin{align*}
	\int \frac{e^{-a|y|}}{|y|^\sigma |x-y|} dy &= \frac{4\pi}{|x|} \int_0^{|x|} e^{-ar} r^{2-\sigma} dr + 4\pi \int_{|x|}^\infty e^{-ar} r^{1-\sigma} dr \\
	&= 4\pi a^{2-\sigma} \Gamma(2-\sigma) + 4\pi \left( \frac{1}{|x|}\int_0^{|x|} e^{-ar} r^{2-\sigma} dr - \int_0^{|x|} e^{-ar} r^{1-\sigma} dr\right).
	\end{align*}
	Consider 
	\[
	f(\lambda) = \frac{1}{\lambda} \int_0^\lambda e^{-ar} r^{2-\sigma} dr - \int_0^\lambda e^{-ar} r^{1-\sigma} dr, \quad \lambda>0.
	\]
	We see that if $0<\sigma<2$, then
	\[
	\lim_{\lambda \rightarrow 0} f(\lambda)= 0.
	\]
	Moreover,
	\[
	f'(\lambda) =- \frac{1}{\lambda^2} \int_0^\lambda e^{-ar} r^{2-\sigma} dr <0, \quad \forall \lambda>0.
	\]
	This shows that $f$ is a strictly decreasing function, hence $f(\lambda) <0$ for all $\lambda>0$. Thus for $x\ne 0$,
	\[
	\int \frac{e^{-a|y|}}{|y|^\sigma |x-y|} dy < 4\pi a^{2-\sigma} \Gamma(2-\sigma).
	\]
	We conclude that
	\[
	\|V\|_{\mathcal{K}} = 4\pi |c| a^{2-\sigma} \Gamma(2-\sigma)
	\]
	which proves \eqref{norm-K-V}.
	\hfill $\Box$
	

\end{document}